\documentclass[a4paper]{amsart}

\usepackage{amsfonts,amssymb,amscd,amsmath,latexsym,amsbsy,enumerate,stmaryrd,a4wide}

\theoremstyle{plain}
\newtheorem{thm}{Theorem}[section]
\newtheorem{cor}[thm]{Corollary}

\newtheorem{lem}[thm]{Lemma}
\newtheorem{prop}[thm]{Proposition}
\newtheorem{Def}[thm]{Definition}

\theoremstyle{remark}
\newtheorem{rem}[thm]{Remark}
\numberwithin{equation}{section}

\begin{document}
\title{Multivariable Wilson polynomials and degenerate Hecke algebras}
\author{Wolter Groenevelt}

\address{Technische Universiteit Delft\\
EWI-DIAM\\
Postbus 5031\\
2600 GA Delft\\
The Netherlands}
\email{w.g.m.groenevelt@tudelft.nl}

\thanks{This research is done at the Korteweg-De Vries Institute for Mathematics at the University of Amsterdam, while supported by the Netherlands Organization for Scientific Research (NWO) for the Vidi-project ``Symmetry and modularity in exactly solvable models.'' The author likes to thank Jasper Stokman for stimulating discussions.}

\date{October 16, 2007}

\keywords{(nonsymmetric) multivariable Wilson polynomials, degenerate double affine Hecke algebra, polynomial representation, difference-reflection operator, orthogonality relations, duality}

\begin{abstract}
We study a rational version of the double affine Hecke algebra associated to the nonreduced affine root system of type $(C^\vee_n,C_n)$. A certain representation in terms of difference-reflection operators naturally leads to the definition of nonsymmetric versions of the multivariable Wilson polynomials. Using the degenerate Hecke algebra we derive several properties, such as orthogonality relations and quadratic norms, for the nonsymmetric and symmetric multivariable Wilson polynomials
\end{abstract}

\maketitle

\section{Introduction}
The Macdonald theory of multivariable orthogonal polynomials associated to root systems can be understood using a specific representation of Cherednik's double affine Hecke algebra (DAHA) in terms of $q$-difference-reflection operators, see e.g.~Cherednik \cite{Ch2} and Macdonald \cite{M}. Koornwinder polynomials \cite{K} generalize the Macdonald polynomials associated to classical root systems. The algebraic structure underlying the Koornwinder polynomials is Sahi's DAHA associated to the nonreduced affine root system of type $(C_n^\vee,C_n)$, see Noumi \cite{N}, Sahi \cite{S1}, \cite{S2}, and Stokman \cite{St}. The Koornwinder polynomials depend, besides the base $q$, on five parameters, corresponding to the number of $W$-orbits in the affine root system of type $(C_n^\vee,C_n)$ with corresponding affine Weyl group $W$. Many families of multivariable orthogonal polynomials can be obtained as limits of the Koornwinder polynomials, see e.g.~\cite{D2} and \cite{SK}. Several of these families have been associated in the literature to degenerate versions of double affine Hecke algebras. For instance, the Heckman-Opdam polynomials \cite{H} are naturally associated to a trigonometric degenerate DAHA \cite{O}.

In this paper we study an algebra $\mathcal H$ which is a rational degeneration of the DAHA of type $(C^\vee_n,C_n)$, and we show that multivariable Wilson polynomials are associated to $\mathcal H$ in a natural way. Multivariable Wilson polynomials are limits of the Koornwinder polynomials (for $q\rightarrow 1$) depending on five limiting parameters, see Van Diejen \cite{D1},\cite{D2}. The one-variable Wilson polynomials \cite{W} are the most general orthogonal polynomials of hypergeometric type, i.e.~all hypergeometric orthogonal polynomials, for instance, the Jacobi polynomials, can be obtained as a limit of the one-variable Wilson polynomials. It may be expected that the multivariable Wilson polynomials play a similar role in the theory of orthogonal polynomials of hypergeometric type associated to root systems. Let us show that Wilson polynomials formally generalize the $BC$-type Heckman-Opdam polynomials. Under suitable conditions on the parameters the multivariable Wilson polynomials are orthogonal on $(\mathrm i \mathbb R_+)^n$ with respect to the weight function
\[
\begin{split}
\Delta_+(x) &=  \prod_{1 \leq j < k \leq n} \frac{ \Gamma(t \pm  x_j \pm x_k) }{ \Gamma(\pm x_j \pm x_k)}   \prod_{j=1}^n \frac{  \Gamma(a \pm x_j) \Gamma(b\pm x_j) \Gamma(c \pm x_j) \Gamma(d \pm x_j)}{ \Gamma(\pm 2x_j) }.
\end{split}
\]
Here we use the notation $\Gamma(\alpha \pm \beta) = \Gamma(\alpha+\beta)\Gamma(\alpha-\beta)$. The weight function $\Delta_+(x)$ may be considered as a generalization of the weight function for the $BC$-type Heckman-Opdam polynomials. Indeed, divide by a factor
\[
\prod_{j=1}^n \Gamma(a+c+(j-1)t)\Gamma(b+c+(j-1)t)\Gamma(a+d+(j-1)t)\Gamma(b+d+(j-1)t),
\]
and substitute
\[
a = b = \frac{ \alpha +1 }{2}, \qquad c=\frac{\beta+1}{2}+\mathrm{i}\gamma, \quad d = \frac{\beta+1}{2}-\mathrm{i}\gamma, \quad x_j = \mathrm i\gamma \sqrt{y_j},
\]
then after applying Stirling's formula we find that in the limit $\gamma \rightarrow \infty$,
\[
\Delta_+(x) \sim
 \frac{\sqrt{y_1\ldots y_n}}{\gamma^n} \prod_{1 \leq j < k \leq n}|y_j-y_k|^{2t} \prod_{j=1}^n y_j^\alpha (1-y_j)^\beta, \qquad y_1,\ldots, y_n \in (0,1),
\]
and $\Delta_+(x)$ vanishes exponentially if $y_j>1$ for any $j$. 

The rational DAHA $\mathcal H$ that we study has a faithful representation in terms of difference-reflection operators, much like the representation of the trigonometric DAHA from \cite{Ch1}. Nonsymmetric versions of the multivariable Wilson polynomials then appear in the representation theory of $\mathcal H$ as the polynomial eigenfunctions of the analogues of the Cherednik operators. Methods from double affine Hecke algebras, see e.g.~\cite{Ch2}, \cite{M}, can now be used to obtain properties of the nonsymmetric multivariable Wilson polynomials, such as orthogonality relations and a duality property. The rank one version of the algebra $\mathcal H$ and the corresponding one-variable Wilson polynomials have been studied in \cite{G}. Let us remark that a four-parameter subfamily of the (symmetric) multivariable Wilson polynomials have been obtained by Zhang \cite{Z} in the context of degenerate Hecke algebras as the Harish-Chandra transform of the $BC$-type Heckman-Opdam polynomials.

The algebra $\mathcal H$ has appeared earlier in the literature; it is isomorphic to a rational generalized DAHA defined by Etingof, Gan and Oblomkov \cite{EGO}. Results in \cite{EGO} then imply that $\mathcal H$ is also closely related to several other algebras appearing in the literature, e.g.~$\mathcal H$ is the spherical subalgebra of a certain deformed preprojective algebra by Gan and Ginzburg \cite{GG}. In this paper we obtain several new properties of the algebra $\mathcal H$. One remarkable property is that it contains a subalgebra $H$ that may be considered as a deformation of $\mathbb C[W]$ with deformed braid-relations. To be more precise, as an algebra $H$ is generated by $T_0,\ldots,T_n$, which satisfy the following braid-type relations:
\begin{align*}
T_i T_{i+1} T_i T_{i+1} + t T_{i} T_{i+1} &= T_{i+1} T_i T_{i+1} T_i  + tT_{i+1} T_{i}, && i=0,n-1,\\
T_i T_{i+1} T_i & = T_{i+1} T_i T_{i+1}, && i \in [1,n-2],\\
T_i T_j &= T_j T_i, && |i-j| \geq 2.
\end{align*}
Here $t$ is some complex parameter. For $t=0$ the above relations are the usual braid relations for $W$. Deformations like this, i.e.~deformations of a group algebra $\mathbb C[G]$, $G$ a Coxeter group, with deformed braid relations, have recently been studied by Etingof and Rains \cite{ER1},\cite{ER2}.

The paper is organized as follows. After some preliminaries on root systems, we define in section \ref{sec:DegAffHeckeAlg} a subalgebra $H$ of the trigonometric DAHA $\mathfrak H$, and it is shown that the subalgebra $H$ is a deformation of the group algebra $\mathbb C[W]$. Then in section \ref{sec:RatDAHA} we define the rational DAHA $\mathcal H$ as a subalgebra of $\mathfrak H$ that is generated by $H$ and a polynomial algebra. We derive several useful properties of $\mathcal H$, such as a PBW-theorem, and we define the analogues of Cherednik's $Y$-operators. In section \ref{sec:representation} we study a faithful representation $\pi$ of $\mathcal H$ in terms of difference-reflection operators on the vector space $\mathcal P$ consisting of polynomials in $n$ variables. The nonsymmetric multivariable Wilson polynomials are used to describe the decomposition of $\mathcal P$ into irreducible $\pi(H)$-modules. In section \ref{sec:nonsymm Wilson} we derive several properties of the nonsymmetric Wilson polynomials. Finally, in section \ref{sec:symm Wilson} we show how well-known properties of the (symmetric) multivariable Wilson polynomials can be obtained from the representation theory of $\mathcal H$.

\section{The degenerate affine Hecke algebra} \label{sec:DegAffHeckeAlg}

In this section we define an algebra $H$ that we consider as a degenerate affine Hecke algebra of type $\widetilde C_n$. Throughout the paper we use the following notation: for integers $m,n$ with $m<n$ we write $[m,n]$ for the set $\{m,m+1,\ldots,n\}$.

\subsection{The affine root system $(C^\vee_n,C_n)$}
We fix an integer $n \geq 2$. Let $\{\epsilon_j\}_{j=1}^n$ be the standard orthonormal basis for $V=\mathbb R^n$ with standard inner product $\langle \cdot,\cdot\rangle$. Let $\widehat V$ denote the vector space of affine linear transformations from $V$ to $\mathbb R$. We identify $\widehat V$ with $V \oplus \mathbb R \delta$ by writing $f \in \widehat V$ as
\[
f(u) = \langle v,u\rangle+ c\delta(u), \qquad u \in V,
\]
for some $v \in V$ and $c \in \mathbb R$, where $\delta:V \rightarrow \mathbb R$ is defined by $\delta(u)=1$ for all $u \in V$. The inner product on $V$ is extended to a bilinear form on $\widehat V$ by
\[
\langle\lambda_1+\mu_1 \delta, \lambda_2+\mu_2 \delta \rangle= \langle\lambda_1,\lambda_2\rangle, \qquad \lambda_1,\lambda_2 \in V, \quad \mu_1,\mu_2 \in \mathbb R.
\]
We define
\[
\begin{split}
a_0=\delta-2\epsilon_1, \qquad a_i = \epsilon_i - \epsilon_{i+1} \ (i=1,\ldots, n-1), \qquad a_n=2\epsilon_n.
\end{split}
\]
The elements $a_i \in \widehat V$ are the simple roots for the affine root system of type $\widetilde C_n$. Let $W$ be the subgroup of $\mathrm{GL}(\widehat V)$ generated by the simple reflections $s_{a_i}=s_i$, $i=0, \ldots, n$, where for $\beta \in \widehat V$ with $\langle \beta, \beta \rangle \neq 0$ the reflection $s_\beta$ is defined by
\[
s_\beta f = f- \langle f,\beta^\vee\rangle \beta, \qquad \beta^\vee= \frac{2\beta}{\langle \beta,\beta \rangle},
\]
for $f \in \widehat V$. The group $W$ is a Coxeter group with relations $s_i^2=1$, $i \in [0,n]$, and
\begin{align*}
s_{i}s_{i+1} s_i s_{i+1} &= s_{i+1} s_i s_{i+1}s_{i}, && i =0,n-1,\\
s_i s_{i+1} s_i &= s_{i+1} s_i s_{i+1}, && i \in [1,n-2],\\
s_i s_j & = s_j s_i, && i,j \in [0,n],\ |i-j| \geq 2.
\end{align*}
The set $\mathcal R_r=W\{a_0,\ldots,a_n\} \subset \widehat V$ is the (reduced) affine root system of type $\widetilde C_n$. The nonreduced affine root system of type $(C^\vee_n,C_n)$ is the root system $\mathcal R=\mathcal R_r^\vee \cup \mathcal R_r$. The root systems $\mathcal R, \mathcal R_r^\vee$ and $\mathcal R_r$ all have $W$ as corresponding affine Weyl group. We denote by $W_0 \subset W$ the (finite) Weyl group generated by $s_1,\ldots,s_n$. Note that the subgroup generated by $s_1,\ldots,s_{n-1}$ is isomorphic to the symmetric group $\mathcal S_n$. The set $\Sigma = W_0\{a_1,\ldots,a_n\}$ is the root system of type $C_n$. We write $\mathcal R^\pm, \mathcal R_r^\pm$, $\Sigma^\pm$ for the positive/negative roots in $\mathcal R$, $\mathcal R_r$, $\Sigma$, respectively, with respect to our choice of simple roots. In particular,
\[
\mathcal R_r^+= \Sigma^+ \cup \{ \alpha \in \mathcal R_r \ | \ \alpha(0)>0\}.
\]

The co-root lattice $Q^\vee$ and the weight lattice $P$ of $\Sigma$ in $V$ can both be identified with $\Lambda = \bigoplus_{i=1}^n \mathbb Z \epsilon_i$. We denote by $\Lambda^+$ the cone of dominant weights, i.e., $\Lambda^+ = \bigoplus_{i=1}^n \mathbb Z_{\geq 0} \omega_i$, where $\omega_i=\epsilon_1+\ldots+\epsilon_i$. There is an alternative description of the affine Weyl group $W$ as the semidirect product
\[
W = W_0 \ltimes \tau(\Lambda),
\]
where $\tau(\lambda)$, $\lambda \in \Lambda$, is the translation operator given by $\tau(\lambda)f = f + \langle \lambda, f \rangle \delta$ for $f \in \widehat V$. The translation operator $\tau(\epsilon_i)$, $i \in [1,n]$, can be expressed in terms of the simple reflections by
\begin{equation} \label{eq:tau}
\tau(\epsilon_i) = s_i \cdots s_{n-1} s_n s_{n-1}\cdots s_1s_0s_1\cdots s_{i-1}.
\end{equation}
This is a reduced expression.

\subsection{A degenerate affine Hecke algebra}

Let $V_{\mathbb C}=V \oplus \mathrm{i}V$ be the complexification of $V$. The finite Weyl group $W_0$ acts on $V_{\mathbb C}$ by reflections as defined in the previous subsection. We get a representation of the affine Weyl group $W$ on $V_{\mathbb C}$ by defining $\tau(\lambda) x = x-\lambda$ for $\lambda \in \Lambda$ and $x \in V_{\mathbb C}$. Let $\mathcal P$ be the algebra consisting of polynomials on $V_{\mathbb C}$, then $W$ acts on $\mathcal P$ by $(wp)(x) = p(w^{-1} x)$, in particular
\[
\begin{split}
(s_0p)(x) &= p(1-x_1,x_2,\ldots,x_n),\\
(s_ip)(x) &= p(x_1,\ldots,x_{i-1},x_{i+1},x_i, x_{i+2},\ldots,x_n),\qquad i \in [1,n-1],\\
(s_np)(x) &= p(x_1,\ldots,x_{n-1},-x_n),
\end{split}
\]
where $x = (x_1,\ldots,x_n)$. We denote by $\mathcal P^{W_0}$ the algebra consisting of polynomials $p \in \mathcal P$ that are invariant under the action of the finite Weyl group $W_0$, i.e., $s_ip = p $ for $i \in [1,n]$. The set of monomials  $\{x^\mu \}_{\mu  \in \mathbb Z_{\geq 0}^n}$ forms a linear basis for $\mathcal P$. Here we use the notation $x^\mu = x_1^{\mu_1} x_2^{\mu_2} \cdots x_n^{\mu_n}$ for $x=(x_1,\ldots,x_n)$ and $\mu_i=\langle \mu, \epsilon_i \rangle$. It will be convenient to label this basis with elements in $\Lambda$. For this we define a bijection $\phi:\mathbb Z \rightarrow \mathbb Z_{\geq 0}$ by
\[
\phi(m) =
\begin{cases}
2m, & m \geq 0,\\
-2m-1, & m <0.
\end{cases}
\]
Now for $\lambda = \sum \lambda_i \epsilon_i \in \Lambda$ we denote
\[
\phi(\lambda) = \sum_{i=1}^n \phi(\lambda_i) \epsilon_i.
\]
Then the set $\{x^{\phi(\lambda)}\}_{\lambda \in \Lambda}$ is a linear basis for $\mathcal P$.

Let $\mathbf k :\mathcal R_r \rightarrow \mathbb C$ be a multiplicity function on $\mathcal R_r$, i.e., $\mathbf k$ is constant on $W$-orbits in $\mathcal R_r$. The function $\mathbf k$ is uniquely determined by its values on $a_0,a_1$ and $a_n$, so we may identify $\mathbf k$ with the ordered $3$-tuple $(k_0, k_1, k_n)$, where $k_i=\mathbf k(a_i)$. Let $\mathcal P_X$ be the algebra $\mathbb C[X_1,\ldots,X_n]$. If $p(x)=\sum_\mu c_\mu x^\mu \in \mathcal P$ is a polynomial, we denote by $p(X)$ the element $\sum_\mu c_\mu X^\mu$ in $\mathcal P_X$, where $X^\mu=X_1^{\mu_1} \cdots X_n^{\mu_n}$ for $\mu_i=\langle \mu,\epsilon_i \rangle$. The degenerate (trigonometric) DAHA $\mathfrak H(\mathbf k)$ associated to the root system $\mathcal R_r$ is the complex associative algebra generated by $W$ (with generators $r_i$, $i=0,\ldots,n)$ and $\mathcal P_X$, with the cross-relations
\[
r_i p(X) - (s_ip)(X) r_i = \frac{k_i}{a_i(X)} \Big( (s_ip)(X) - p(X) \Big), \qquad i \in [0,n].
\]
If we say that an algebra $A$ is generated by the algebras $A'$ and $A''$ (both having a unit element) we mean the following:
\begin{itemize}
\item $A \simeq A' \otimes A''$ as vector spaces,
\item for $a' \in A'$ and $a'' \in A''$ the maps $a' \mapsto a'\otimes 1$ and $a'' \mapsto 1 \otimes a''$ are algebra homomorphisms,
\item $(a' \otimes 1)(1 \otimes a'') = a' \otimes a''$, for $a' \in A'$, $a''\in A''$.
\end{itemize}
Moreover, we will write $a'a''$ instead of $a' \otimes a''$.

The algebra $\mathfrak H(\mathbf k)$ has a linear basis, called the Poincar\'e-Birkhoff-Witt (PBW) basis,
\[
\{ wX^{\phi(\lambda)} \ | \ w \in W, \ \lambda \in \Lambda \}.
\]

Let $\mathbf t:\mathcal R \rightarrow \mathbb C$ be a multiplicity function on $\mathcal R$, that we identify with the ordered 5-tuple $(\mathbf t(a_0), \mathbf t(a_0^\vee), \mathbf t(a_1), \mathbf t(a_n), \mathbf t(a_n^\vee))=(t_0, u_0, t, t_n, u_n)$. We assume that the function $\mathbf t$ is related to the multiplicity function $\mathbf k$ by
\[
k_0 = 2t_0+2u_0, \quad k_1=t, \quad k_n = 2t_n+2u_n.
\]
We write $\mathbf t_{\mathcal R_r}=(t_0,t,t_n)$ and $\mathbf t_{\mathcal R_r^\vee} = (u_0,t,u_n)$.
We now define a subalgebra $H(\mathbf t_{\mathcal R_r})$ of $\mathfrak H(\mathbf k)$ that may be considered as a rational version of the affine Hecke algebra of type $\widetilde C_n$.
\begin{Def}
The algebra $H=H(\mathbf t_{\mathcal R_r}) \subset \mathfrak H(\mathbf k)$ is the subalgebra  generated by $r_1,\ldots,r_{n-1}$ and
\[
T_i = t_i + \frac12(t_i-u_i+a_i^\vee(X))(r_i-1), \qquad i=0,n.
\]
\end{Def}
Although it is not clear at this point, the algebra $H$ is an $(u_0,u_n)$-dependent subalgebra of $\mathfrak H(\mathbf k)$ that depends as an abstract algebra only on $\mathbf t_{\mathcal R_r}=(t_0,t,t_n)$. When working in the algebra $H$ we will write $T_i=r_i$ for $i \in [1,n-1]$, so $H$ is the algebra generated by $T_i$, $i\in [0,n]$. We also define the algebra $H_0\subset H$ to be the subalgebra generated by $T_i$, $i \in [1,n]$. This is the analogue of the finite Hecke algebra of type $C_n$.

\begin{prop} \label{prop:relations H}
We have the following relations in the algebra $H$:\\
Quadratic relations:
\begin{align*}
T_i^2&=t_i^2, && i=0,n,\\
T_i^2&=1, && i\in[1,n-1],
\end{align*}
Braid-type relations:
\begin{align*}
T_i T_{i+1} T_i T_{i+1} + t T_{i} T_{i+1} &= T_{i+1} T_i T_{i+1} T_i  + tT_{i+1} T_{i}, && i=0,n-1,\\
T_i T_{i+1} T_i & = T_{i+1} T_i T_{i+1}, && i \in [1,n-2],\\
T_i T_j &= T_j T_i, && i,j \in [0,n],\ |i-j| \geq 2.
\end{align*}
\end{prop}
\begin{proof}
We only need to verify the relations involving $T_0$ and $T_n$, the other ones follow directly from the definition of the algebra $H$. For the quadratic relations we find from the definition of $T_i$, $i=0,n$, and the relations in $\mathfrak H(\mathbf k)$,
\[
\begin{split}
[r_i-1][t_i-u_i+a_i^\vee(X)][r_i-1] &= [(t_i-u_i-a_i^\vee(X))r_i-3t_i-u_i-a_i^\vee(X)][r_i-1]\\
& =-4t_i(r_i-1)
\end{split}
\]
which gives
\[
\begin{split}
T_i^2=&t_i^2 + t_i[t_i-u_i+a_i^\vee(X)][r_i-1] \\
&+ \frac14[t_i-u_i+a_i^\vee(X)]\Big([r_i-1][t_i-u_i+a_i^\vee(X)][r_i-1]\Big)\\
 = &t_i^2.
\end{split}
\]

Next we check the braid-type relation for $T_n$. The braid-type relation for $T_0$ is checked in the same way. Let us write
\[
 f(y) = \frac12(t_n-u_n-y), \quad g(y) = \frac12(t_n+u_n+y),
\]
then it follows from the relations in $\mathfrak H(\mathbf k)$ that
\[
 T_n = t_n + (s_nf)(X_n)(r_n-1)= r_nf(X_n) - g(X_n).
\]
Note that we also have
\[
\begin{split}
f(X_n) r_{n-1} &= r_{n-1} f(X_{n-1}) -t/2,\\
g(X_n) r_{n-1} & = r_{n-1} g(X_{n-1}) + t/2.
\end{split}
\]
Now we find in $\mathfrak H(\mathbf k)$
\[
T_n r_{n-1} = r_n r_{n-1} f(X_{n-1}) -\frac12 t r_n - r_{n-1} g(X_{n-1}) -t/2.
\]
Multiplying this from the right by $T_n = r_n f(X_n) - g(X_n)$ we obtain, after some calculations
\begin{equation} \label{eq:h1}
\begin{split}
T_n r_{n-1} T_{n} +tT_n  =& r_n r_{n-1} r_n f(X_{n-1}) f(X_n) - r_n r_{n-1} f(X_{n-1}) g(X_n)\\
&- r_{n-1} r_n g(X_{n-1}) f(X_n) + r_{n-1} g(X_{n-1}) g(X_n) \\
&+ \frac12 t (r_n-1)[f(X_n)+g(X_n)].
\end{split}
\end{equation}
Note that $f(y)+g(y) = t_n$. Now we multiply \eqref{eq:h1} from the right by $r_{n-1}$ and bring $r_{n-1}$ to the left of $f$ and $g$, then we have after a few calculations
\[
\begin{split}
 T_n r_{n-1} T_{n}r_{n-1} &+tT_n r_{n-1}  = \\
&r_n r_{n-1} r_n r_{n-1} f(X_n) f(X_{n-1}) - r_{n-1} r_n r_{n-1} g(X_n) f(X_{n-1})\\
& + \frac12 t_n t r_{n-1} r_n -r_n f(X_n) g(X_{n-1})- \frac12 t_n t r_{n-1} +g(X_n) g(X_{n-1}) .
\end{split}
\]
By inspection, using the braid relations for $r_n$, this is the same as \eqref{eq:h1} multiplied by $r_{n-1}$ from the left.
\end{proof}
The previous proposition shows that $H$ can be considered as a deformation of the group algebra $\mathbb C[W]$, with deformed braid-relations for $i=0,n$.

In order to show that the relations from Proposition \ref{prop:relations H} characterize $H$ as an algebra we show that $H$ has a PBW-basis. We introduce the following notation. If $u$ is a word in the $s_i$'s, $u=s_{i_1}\cdots s_{i_r}$, then we write $T_u=T_{i_1} \cdots  T_{i_r}$.
\begin{prop}[PBW-property for $H$] \label{prop:basis T}
For every $w \in W$ let $\overline{w}$ be a fixed reduced expression for $w$, then the set $\{T_{\overline w} \ |\ w \in W \}$ is a linear basis for $H$.
\end{prop}
\begin{proof}
The set $\{T_{\overline w} \ | \ w\in W \}$ spans $H$, see \cite[Theorem 2.3]{ER1}.
Now suppose that we have a relation $\sum_{u \in W} c_u T_{\overline{u}} =0 $ in $H$ for some coefficients $c_u \in \mathbb C$. Using the cross-relations in $\mathfrak H(\mathbf k)$, we may write
\[
T_{\overline{u}} = \sum_{\substack{w \in W\\w \leq u}} f_{u,w}(X)w,
\]
where $f_{u,w}(X) \in \mathcal P_X$ and $f_{u,u}$ is nonzero. Here $w \leq u$ is meant with respect to the Bruhat order on $W$. Now we have
\[
\sum_{\substack{u,w \in W\\w \leq u}} c_u f_{u,w}(X) w =0,
\]
which implies, by the PBW-property for $\mathfrak H$, that for any $w \in W$
\[
\sum_{u \geq w} c_u f_{u,w}(X) =0.
\]
Let $v$ be a maximal element (in the Bruhat order) in the set $\{u \in W \ | \ c_u \neq 0\}$, then $c_v f_{v,v}(X)=0$ in $\mathcal P_X$. But since $f_{v,v}$ is a nonzero polynomial, it follows that $c_v=0$. We conclude that $c_u=0$ for all $u \in W$.
\end{proof}
\begin{cor} \label{cor:basis H}
As an algebra $H$ is completely characterized by the relations from Proposition \ref{prop:relations H}.
\end{cor}
\begin{proof}
Suppose that we have a complex associative algebra $\mathcal V$ generated by $V_i$, $i=0,\ldots,n$, with the same relations as in Proposition \ref{prop:relations H} (with $T_i$ replaced by $V_i$). Then the assignments $V_i \mapsto T_i$ extend to a surjective homomorphism $\psi:\mathcal V \rightarrow H$. Now let $V \in \mathcal V$ satisfy $\psi(V)=0$. By \cite[Theorem 2.3]{ER1} the set $\{V_{\overline w} \ | \ w \in W\}$ spans $V$. Writing $V= \sum_{w \in W} c_w V_{\overline{w}}$ and applying $\psi$ shows that $\sum_{w \in W} c_w T_{\overline{w}}=0$, hence every $c_w$ is equal to zero by Proposition \ref{prop:basis T}. This shows that $\psi$ is also injective.
\end{proof}
We also have the PBW-property for the finite algebra $H_0$.
\begin{cor}
The set $\{T_{\overline w} \ |\ w \in W_0 \}$ is a linear basis for $H_0$.
\end{cor}
From Proposition \ref{prop:relations H} and Corollary \ref{cor:basis H} it is now clear that $H$ depends as an abstract algebra only on the parameters $t_0,t,t_n$. Naturally we also have an algebra $H(\mathbf t_{\mathcal R_r^\vee})$ related to the reduced root system $\mathcal R_r^\vee$ which is obtained from the algebra $H(\mathbf t_{\mathcal R_r})$ by replacing $(t_0,t_n)$ by $(u_0,u_n)$.

\section{The rational double affine Hecke algebra} \label{sec:RatDAHA}
We now come to the definition of the algebra $\mathcal H(\mathbf t)$ that may be considered as a rational degeneration of Sahi's double affine Hecke algebra for type $(C^\vee, C)$. This is explained in Remark \ref{rem:limits}. We call this algebra the rational double affine Hecke algebra. For rank 1 this algebra is studied in \cite{G}. Let us remark that $\mathcal H$ is not a rational Cherednik algebra as defined in \cite{EG}. We show in subsection \ref{ssec:rational GDAHA} that the algebra $\mathcal H(\mathbf t)$ is isomorphic to the rational generalized double affine Hecke algebra attached to an affine Dynkin diagram of type $\widetilde D_4$ as defined by Etingof, Gan and Oblomkov in \cite{EGO}.
\begin{Def}
The algebra $\mathcal H=\mathcal H(\mathbf t)\subset \mathfrak H(\mathbf k)$ is the subalgebra generated by $H(\mathbf t_{\mathcal R_r})$ and $\mathcal P_X$.
\end{Def}
From the relations in $\mathfrak H(\mathbf k)$ it follows that in $\mathcal H$ we have, for $p \in \mathcal P$,
\begin{equation} \label{eq:relTPX}
T_i p(X)-(s_ip)(X) T_i = d_i(X;\mathbf t)\big[(s_ip)(X)-p(X)\big],\qquad i \in [0,n],
\end{equation}
where
\[
d_i(X;\mathbf t) =
\begin{cases}
\dfrac{t_i^2-u_i^2+a_i^\vee(X)^2}{a_i(X)},& i=0,n,\\
\dfrac{t}{a_i(X)},& i \in [1,n-1].
\end{cases}
\]
Together with the PBW-Theorem for the algebra $H$ this leads to the PBW-Theorem for $\mathcal H$.
\begin{prop}[PBW property for $\mathcal H$]
The sets
\[
\{ X^{\phi(\lambda)} T_{\overline{w}} \ | \ \lambda\in \Lambda, \ w \in W\}, \qquad \{  T_{\overline{w}}X^{\phi(\lambda)} \ | \ \lambda\in \Lambda, \ w \in W\}
\]
form linear bases for $\mathcal H$.
\end{prop}
\begin{cor} \label{cor:defrel H}
The algebra $\mathcal H$ is completely characterized as an algebra by the relations for $H$ from Proposition \ref{prop:relations H}, the relations $X_i X_j = X_j X_i$ for $i,j\in [1,n]$, and the cross relations \eqref{eq:relTPX}.
\end{cor}
We will frequently use \eqref{eq:relTPX} with $p(X)=X_i$, $i \in [1,n]$. In fact, \eqref{eq:relTPX} can be recovered from these relations. We collect the identities in the following lemma.
\begin{lem} \label{lem:relations daha}
In $\mathcal H$ the following relations hold:
\begin{align*}
(X_1-\frac12-T_0)^2&=u_0^2, \\
(X_n+T_n)^2&=u_n^2,\\
T_i X_i T_i&= X_{i+1} -tT_i, && i \in [1,n-1],\\
X_i T_j &= T_j X_i, && i \in [1,n],\ j \in [0,n],\ |i-j| \geq 2,  \\
X_iT_{i+1} &= T_{i+1}X_i, && i \in [1,n-1].
\end{align*}
\end{lem}
We define in $\mathcal H$
\[
T_0^\vee = X_1-\frac12-T_0, \qquad T_n^\vee = -X_n-T_n,
\]
then we have $(T_0^\vee)^2=u_0^2$ and $(T_n^\vee)^2=u_n^2$. For $i\in [1,n-1]$ we will sometimes denote $T_i^\vee = T_i$. Let us introduce some convenient notations. We write
for $i,j \in [1,n]$, $i< j$,
\[
\begin{split}
T_{i,j} &= T_i T_{i+1} \cdots T_{j-2} T_{j-1} T_{j-2}\cdots T_{i+1} T_i \\
& = T_{j-1} T_{j-2} \cdots T_{i+1} T_i T_{i+1} \cdots T_{j-2}  T_{j-1}.
\end{split}
\]
This notation corresponds to the familiar notation in the symmetric group $\mathcal S_n$ where $s_{ij}$ denotes the transposition $i \leftrightarrow j$. Recall here that the elements $T_i$, $i=1,\ldots,n-1$, generate a subalgebra isomorphic to $\mathcal S_n$. We also define
in $\mathcal H$
\begin{align*}
\Xi_{i,n}&=T_i T_{i+1} \cdots T_{n-1} T_n T_{n-1} \cdots T_{i+1} T_{i},\\
\Xi_{i,n}^\vee & = T_i T_{i+1}\cdots T_{n-1} T_n^\vee T_{n-1} \cdots T_{i+1} T_{i},\\
\Xi_{0,i} & = T_{i-1} T_{i-2}\cdots T_1 T_0 T_1 \cdots T_{i-2} T_{i-1},\\
\Xi_{0,i}^\vee &= T_{i-1} T_{i-2}\cdots T_1 T_0^\vee T_1 \cdots T_{i-2} T_{i-1},
\end{align*}
for $i \in [1,n]$.

We give a few useful relations in $\mathcal H$ involving $T_0^\vee$ and $T_n^\vee$.
\begin{lem} \label{lem:relations Tvee}
The following relations hold in $\mathcal H$:
\begin{subequations}
\begin{align}
T_0 T_{1} T_0^\vee T_{1} & = T_{1}T_0^\vee T_{1} T_0 \label{eq:rela}\\
T_n T_{n-1} T_n^\vee T_{n-1} &= T_{n-1}T_n^\vee T_{n-1} T_n \label{eq:relb}\\
T_0^\vee X_i &= X_i T_0^\vee, && i \in [2,n] \label{eq:relc}\\
T_n^\vee X_i &= X_i T_n^\vee, &&i \in [1,n-1] \label{eq:reld}\\
T_0T_n^\vee &= T_n^\vee T_0 \label{eq:rele}\\
T_0^\vee T_n &= T_n T_0^\vee \label{eq:relf}\\
T_0^\vee T_n^\vee &= T_n^\vee T_0^\vee \label{eq:relg}\\
\frac12+T_0 + T_0^\vee + \Xi_{1,n} &+ \Xi_{1,n}^\vee + t \sum_{j=2}^{n} T_{1,j} =0. \label{eq:relh} \
\end{align}
\end{subequations}
\end{lem}
\begin{proof}
Relations \eqref{eq:relc}-\eqref{eq:relg} are obvious. For \eqref{eq:relb} use the relations $T_{n-1} X_n T_{n-1} = X_{n-1} + tT_{n-1}$, $T_n X_{n-1}= X_{n-1} T_n$ from Lemma \ref{lem:relations daha} and the braid-type relations for $T_n$;
\[
\begin{split}
T_n T_{n-1} (X_n+ T_n) T_{n-1} &= T_n (X_{n-1}+tT_{n-1}) + T_{n-1} T_n T_{n-1} T_n + tT_{n-1} T_n - tT_n T_{n-1}\\
&= (X_{n-1}+tT_{n-1}) T_n +T_{n-1} T_n T_{n-1} T_n \\
&= T_{n-1} (X_n+T_n) T_{n-1} T_n.
\end{split}
\]
Relation \eqref{eq:rela} follows in the same way.
In order to prove \eqref{eq:relh}, we observe that from $T_iX_{i+1}T_i = X_{i} +tT_i$ for $i\in [1,n-1]$, we find by backward induction
\[
T_1 \cdots T_{n-1} X_n T_{n-1} \cdots T_1 = X_1+t \sum_{j=2}^{n} T_{1,j}.
\]
Then \eqref{eq:relh} follows from the identities $T_0^\vee+T_0+\frac12=X_1 $ and $T_n+T_n^\vee = -X_n$.
\end{proof}
Next we show that $\mathcal H$ has a subalgebra isomorphic to the degenerate affine Hecke algebra $H(\mathbf t_{\mathcal R_r^\vee})$.
\begin{prop} \label{prop:algebra Hv}
The subalgebra of $\mathcal H$ generated by $T_j^\vee$, $j=0,\ldots,n$, is isomorphic to $H(\mathbf t_{\mathcal R_r^\vee})$, i.e., it is characterized by the following relations:\\

\noindent Quadratic relations:
\begin{align*}
(T_i^\vee)^2&=u_i^2, && i=0,n, \\
(T_i^\vee)^2&=1, &&i\in[1,n-1],
\end{align*}
Braid-type relations:
\begin{align*}
T^\vee_i T^\vee_{i+1} T^\vee_i T^\vee_{i+1} + t T^\vee_i T^\vee_{i+1} &= T^\vee_{i+1} T^\vee_i T^\vee_{i+1} T^\vee_i + t T^\vee_{i+1} T^\vee_j, && i=0,n-1,\\
T^\vee_i T^\vee_{i+1} T^\vee_i &= T^\vee_{i+1} T^\vee_i T^\vee_{i+1},&& i\in[1,n-2],\\
T^\vee_i T^\vee_j &= T^\vee_j T^\vee_i, &&  i,j\in[0,n],\ |i-j| \geq 2.\\
\end{align*}
\end{prop}
\begin{proof}
Let $H^\vee$ be the subalgebra generated by $T_i^\vee$, $i=0,\ldots,n$. Let us first verify that the relations mentioned in the proposition are satisfied in $H^\vee$. We only need to check the braid-type relations involving $T_0^\vee$ and $T_{n}^\vee$. Let us concentrate on the relations for $T_n^\vee$. We already know from Lemma \ref{lem:relations Tvee} that $T_i T_n^\vee = T_n^\vee T_i$ for $i\in[1,n-2]$. The relation involving $T_{n-1}$ follows from writing out $X_n X_{n-1} = X_{n-1} X_n$ if we write $X_{n-1}=T_{n-1}X_n T_{n-1} + tT_{n-1}$, $X_{n}=-T_n-T_n^\vee$, and we use the braid-type relation for $T_n$ and \eqref{eq:relb}. For $T_0^\vee$ the proof is similar.

In order to show that the $H^\vee$ is indeed isomorphic to $H(\mathbf t_{\mathcal R_r^\vee})$ it is now enough to show that $\{ T_{\overline w}^\vee \ | \ w \in W\}$ is a PBW-basis for $H^\vee$. This is done in the same way as in Proposition \ref{prop:basis T}.
\end{proof}
We now come to the following characterization of the rational DAHA $\mathcal H$, which is the analogue of \cite[Theorem 3.4]{St}.
\begin{thm} \label{thm:daha=H Hv}
The algebra $\mathcal H$ is isomorphic to the complex associative algebra $\mathcal V$ generated by $V_0,V_0^\vee, V_n,V_n^\vee, V_i=V_i^\vee$, $i=1,\ldots,n-1$, such that
\begin{enumerate}[(i)]
\item the subalgebra generated by $V_i$ ($i=0,\ldots,n$) is isomorphic to $H(\mathbf t_{\mathcal R_r})$,
\item the subalgebra generated by $V_i^\vee$ ($i=0,\ldots,n$) is isomorphic to $H(\mathbf t_{\mathcal R_r^\vee})$,
\item the following compatibility relations are satisfied:
 \begin{gather}
[V_0^\vee, V_n] = 0,  \quad [V_n^\vee, V_0] = 0, \nonumber \\
\frac12+V_0 + V_0^\vee + \Upsilon_{1,n} + \Upsilon_{1,n}^\vee + t \sum_{j=2}^{n} V_{1,j} =0,  \label{eq:compatibility}
\end{gather}
where
\[
\begin{split}
V_{1,j} &= V_1 V_2 \cdots V_{j-2} V_{j-1} V_{j-2} \cdots V_2 V_1, \qquad j \in [2,n],\\
\Upsilon_{1,n} &=V_1 V_2 \cdots V_{n-1} V_{n} V_{n-1} \cdots V_2 V_1, \\
\Upsilon_{1,n}^\vee &= V_1 V_2 \cdots V_{n-1} V_n^\vee V_{n-1} \cdots V_2 V_1.
\end{split}
\]
\end{enumerate}
The isomorphism $\varphi:\mathcal V \rightarrow \mathcal H$ is given on the generators of $\mathcal V$ by
\[
\varphi(V_i)= T_i, \quad \varphi(V_i^\vee)= T_i^\vee,
\]
for $i \in [0,n]$.
\end{thm}
We prove the theorem in the next subsection.
\subsection{Proof of Theorem \ref{thm:daha=H Hv}}
First note that $\varphi$ preserves the defining relations for $\mathcal V$, and $\varphi$ maps generators of $\mathcal V$ to generators of $\mathcal H$, hence $\varphi$ is a surjective algebra homomorphism. In order to prove that $\varphi$ is also injective, we show that we can recover the defining relations for $\mathcal H$ in terms of $T_0,
\ldots, T_n$, $X_1,\ldots,X_n$, see Corollary \ref{cor:defrel H}, from the relations in $\mathcal V$. Since the elements $V_i$, $i=0,\ldots,n$, generate an algebra isomorphic to $H(\mathbf t_{\mathcal R_r})$, we need to find pairwise commuting elements $v_i \in \mathcal V$, $i=1,\ldots,n$, such that $\varphi(v_i)=X_i$, and the following cross-relations holds, see \eqref{eq:relTPX}:
\begin{equation} \label{eq:relVv}
V_i p(v)-(s_ip)(v) V_i = d_i(v;\mathbf t)\big[(s_ip)(v)-p(v)\big],\qquad i \in [0,n].
\end{equation}
Here we use the usual notation $p(v)=\sum_\mu c_\mu v^\mu$, if $p(x)$ is the polynomial $\sum_\mu c_\mu x^\mu$, and $v^\mu = v_1^{\mu_1} \cdots v_n^{\mu_n}$ for $\mu \in \mathbb Z_{\geq 0}^n$.\\

Let us first introduce the elements $v_i$. We define in $\mathcal V$
\[
v_n = -V_n - V_n^\vee,
\]
and we define $v_i$ recursively by
\begin{equation} \label{eq:recursionv}
v_i  = V_i v_{i+1} V_i - t V_i, \qquad i \in [1,n-1],
\end{equation}
then $\varphi(v_i)=X_i$ for $i \in [1,n]$. Note that we may write $v_i$ explicitly as
\begin{equation} \label{def:vi}
\begin{split}
v_i &= -\Big( \Upsilon_{i,n} + \Upsilon_{i,n}^\vee + t \sum_{j=i+1}^{n} V_{i,j}\Big),  \qquad i \in [1,n],
\end{split}
\end{equation}
where
\begin{align*}
V_{i,j}&=V_{i} V_{i+1} \cdots V_{j-2}V_{j-1}V_{j-2} \cdots V_{i+1}V_{i}, && 1 \leq i < j \leq n,\\
\Upsilon_{i,n} &= V_{i} V_{i+1} \cdots V_{n-1}V_{n}V_{n-1} \cdots V_{i+1}V_{i}, && i \in [1,n],\\
\Upsilon_{i,n}^\vee &= V_{i} V_{i+1} \cdots V_{n-1}V_{n}^\vee V_{n-1} \cdots V_{i+1}V_{i}, && i \in [1,n].\\
\end{align*}
For $v_1$ the compatibility relation \eqref{eq:compatibility} gives
\[
v_1 = \frac12 +V_0+V_0^\vee,
\]
so using the recursion relation for $v_i$ again we also have
\begin{equation} \label{def:vi2}
v_i = \frac12+\Upsilon_{0,i}+\Upsilon_{0,i}^\vee+t\sum_{j=2}^{i} V_{1,j}, \qquad i \in [1,n],
\end{equation}
where
\begin{align*}
\Upsilon_{0,i} &= V_{i-1} V_{i-2}\cdots V_{1} V_0 V_1 \cdots V_{i-2} V_{i-1}, && i \in [1,n],\\
\Upsilon_{0,i}^\vee &= V_{i-1} V_{i-2}\cdots V_{1} V_0^\vee V_1 \cdots V_{i-2} V_{i-1}, && i \in [1,n].
\end{align*}
We can now show that we have all the relations from Lemma \ref{lem:relations daha} in $\mathcal V$.
\begin{lem} \label{lem:vV}
The following relations hold in $\mathcal V$:
\begin{align*}
(v_1-\frac12-V_0)^2 &= u_0^2,\\
(v_n+V_n)^2&=u_n^2,\\
V_i v_i V_i &= v_{i+1}-tV_i, && i \in [1,n-1],\\
v_i V_j &= V_j v_i, && j \in [0,n], \ i \in [1,n], \ |i-j| \geq 2,\\
v_i V_{i+1} &= V_{i+1} v_i, && i \in [1,n-1].\\
\end{align*}
\end{lem}
\begin{proof}
The first three relations follow directly from the definition of $v_i$.
Fix $i \in [2,n]$. In $\mathcal V$ we have
\[
\begin{split}
 [U_1,U_2] = 0, \qquad& U_1 \in \{V_0^\vee,V_0,V_1,\ldots,V_{i-2}\},\\
& U_2 \in \{V_{i,i+1}, V_{i,i+2},\ldots,\Upsilon_{i,n}, \Upsilon_{i,n}^\vee\}.
\end{split}
\]
These relations are easily verified using $[V_k,V_l]=[V_0^\vee,V_l]=[V_k,V_n^\vee]=0$ for $k \in [0,j]$, $l \in [i+1,n]$. Now from \eqref{def:vi} we find for $j \in [0,i-2]$
\[
\begin{split}
V_j v_i &= -\Big( V_j\Upsilon_{i,n} + V_j\Upsilon_{i,n}^\vee + t \sum_{k=i+1}^{n}V_j V_{i,k}\Big)\\
&= -\Big( \Upsilon_{i,n} V_j +  \Upsilon_{i,n}^\vee V_j+ t \sum_{k=i+1}^{n} V_{i,k}V_j\Big)\\
&= v_iV_j.
\end{split}
\]

Next we fix $i \in [1,n-1]$. Now we use the relations
\[
\begin{split}
[U_1,U_2]=0, \qquad &U_1 \in \{V_{i+1},V_{i+2},\ldots,V_{n},V_n^\vee\}, \\
& U_2 \in \{\Upsilon_{0,i}, \Upsilon_{0,i}^\vee, V_{1,2}, V_{1,3},\ldots, V_{1,i}\},
\end{split}
\]
then by \eqref{def:vi2} we have for $j \in [i+1,n]$
\[
\begin{split}
V_j v_i &= \frac12 V_j+V_j\Upsilon_{0,i}+V_j\Upsilon_{0,i}^\vee+t\sum_{k=2}^{i} V_j V_{1,k}\\
&= \frac12 V_j+\Upsilon_{0,i}V_j+\Upsilon_{0,i}^\vee V_j+t\sum_{k=2}^{i}  V_{1,k}V_j\\
&= v_iV_j. \qedhere
\end{split}
\]
\end{proof}
We still need to show that the subalgebra generated by $v_i$, $i=1,\ldots,n$, is commutative. Once the commutativity of the $v_i$'s is established, the cross-relations \eqref{eq:relVv} follow from Lemma \ref{lem:vV}. The following lemma will be useful for proving the commutativity of the $v_i$'s.
\begin{lem} \label{lem:vV1}
We have the following relations in $\mathcal V$:
\begin{align*}
v_i V_0^\vee &= V_0^\vee v_i,&& i \in[2,n],\\
v_i V_n^\vee &= V_n^\vee v_i,&& i \in[1,n-1], \\
V_n V_{n-1} V_n^\vee V_{n-1} &= V_{n-1} V_n^\vee V_{n-1}V_n
\end{align*}
\end{lem}
\begin{proof}
For the first two relations see the proof of Lemma \ref{lem:vV}. The last relation follows from writing out $[v_{n-1},V_n]=0$ in terms of $V_{n-1},V_n$ and $V_n^\vee$, and using the braid-type relations for $V_n$ and $V_{n-1}$.
\end{proof}
We are now ready to prove the commutativity of the $v_i$'s, which completes the proof of Theorem \ref{thm:daha=H Hv}.
\begin{lem}
For $i,j \in [1,n]$, we have $v_i v_j = v_j v_i$.
\end{lem}
\begin{proof}
We first show that $v_{n-1}v_{n} = v_{n}v_{n-1}$. By \eqref{eq:recursionv} and the definition of $v_n$ we have
\[
\begin{split}
v_{n-1} v_{n} &= V_{n-1} v_n V_{n-1} v_n - t V_{n-1} v_n \\
&= V_{n-1}(V_n+V_n^\vee)V_{n-1}(V_n+V_n^\vee)+t V_{n-1}V_n + tV_{n-1}V_n^\vee.
\end{split}
\]
By Lemmas \ref{lem:vV}, \ref{lem:vV1} and the braid-type relations for $V_n$ and $V_n^\vee$ we obtain from this
\[
\begin{split}
v_{n-1} v_{n} &= (V_n+V_n^\vee)V_{n-1}(V_n+V_n^\vee)V_{n-1}+t V_n V_{n-1}+ tV_n^\vee V_{n-1} \\
&= v_n V_{n-1} v_n V_{n-1} -t v_n V_{n-1} \\
&= v_n v_{n-1}.
\end{split}
\]

Next, we show that $v_n v_i = v_i v_n$ by backward induction on $i$. The statement is true for $i=n-1$. Suppose that $v_n v_{i+1}=v_{i+1}v_n$ for some $i \in [1,n-2]$, then it follows from \eqref{eq:recursionv} and Lemma \ref{lem:vV} that
\[
\begin{split}
v_n v_i &= v_n V_i v_{i+1} V_i - tv_n V_i \\
&= V_i v_n v_{i+1} V_i - t V_i v_n \\
&= V_i v_{i+1} v_n V_i - t V_i v_n \\
&= V_i v_{i+1} V_i v_n -t V_i v_n\\
&= v_i v_n.
\end{split}
\]
This proves the induction step.

Finally, let $i \in [1,n-1]$. A similar induction argument as above shows that $v_i v_j= v_j v_i$ for all $j \in [i+1,n]$.
\end{proof}

\subsection{The duality isomorphism}
We define an involution $\sigma$ on the multiplicity function $\mathbf t$ by
\[
\mathbf t^\sigma=(u_n, u_0, t, t_n, t_0),
\]
i.e., the values of $\mathbf t$ on the $a_0$-orbit and the $a_n^\vee$-orbit are interchanged. If an object depends on $\mathbf t$ we attach a superscript or subscript $\sigma$ to denote the same object depending on $\mathbf t^\sigma$, for example, $\mathcal H_\sigma=\mathcal H(\mathbf t^\sigma)$ and we denote the elements in this algebra by $Z^\sigma$ ($Z \in \mathcal H$).

For ease of notations, we will write
\[
U_0 = T_0^\vee, \qquad U_n= \Xi_{1,n}^\vee.
\]
Note that $U_i^2=u_i^2$, and $U_0,U_1,T_i$, $i \in [0,n]$, form a set of generators for the algebra $\mathcal H$.
\begin{prop} \label{prop:duality iso}
The assignments
\[
\begin{split}
U_0 \mapsto U_0^\sigma, \quad U_n \mapsto T_0^\sigma, \quad T_0 & \mapsto U_n^\sigma,  \quad T_i \mapsto T_i^\sigma\ (i \in [1,n]),
\end{split}
\]
extend uniquely to an isomorphism $\sigma : \mathcal H \rightarrow \mathcal H_\sigma$, respectively an anti-isomorphism $\psi:\mathcal H \rightarrow \mathcal H_\sigma$.
\end{prop}
To prove Proposition \ref{prop:duality iso} we only need to check that $\sigma$ and $\psi$ preserve the defining relations for $\mathcal H$, which is a straightforward exercise. The only relation that is not obvious is the braid-type relation for $U_n$ and $T_1$;
\[
U_n T_1 U_n T_1 + t U_n T_1 = T_1 U_n T_1 U_n + t T_1 U_n.
\]
We prove a slightly more general result which is useful later on. The relation with $U=\Xi_{1,n}^\vee$ and $j=2$ is the desired braid-type relation for $U_n$ and $T_1$.
\begin{lem} \label{lem:braidrel}
For $j\in [2,n]$ the following braid-type relations hold in $\mathcal H$:
\[
UT_{1,j}UT_{1,j} + t U T_{1,j}= T_{1,j}UT_{1,j}U + t  T_{1,j}U, \qquad U \in \{T_0, T_0^\vee, \Xi_{1,n}, \Xi_{1,n}^\vee\}.
\]
\end{lem}
\begin{proof}
Let us first prove the identity for $U=\Xi_{1,n}$. The identity for $U=\Xi_{1,n}^\vee$ is proved in exactly the same way. Using $T_n = T_{1,n} \Xi_{1,n} T_{1,n}$ and $T_{1,n}T_{1,j}T_{1,n} = T_{j,n}$ we find
\[
\Xi_{1,n}T_{1,j}\Xi_{1,n}T_{1,j} + t \Xi_{1,n} T_{1,j}= T_{1,n}\big(  T_n T_{j,n} T_nT_{j,n} + t T_n T_{j,n}\big)T_{1,n}.
\]
We write out $T_{j,n}$ in terms of $T_i$'s, then by repeated application of $[T_i,T_j]=0$ and $T_i^2=1$ the expression between brackets becomes
\[
\begin{split}
T_n (T_{j} \cdots T_{n-1} \cdots T_{j} )T_n (T_j \cdots T_{n-1} \cdots T_j) + t T_n (T_j \cdots T_{n-1}\cdots T_j) = \\
T_j \cdots T_{n-2}(T_nT_{n-1} T_n  T_{n-1}  + tT_n T_{n-1}) T_{n-2} \cdots T_j.
\end{split}
\]
Now use the braid-type relations between $T_n$ and $T_{n-1}$ and reverse the above steps, then we obtain the desired identity.

For $U=T_0,T_0^\vee$ the desired identities for $j=2$ are precisely the braid-type relation between $U$ and $T_1$. For $j \in [3,n]$ the desired identity follow easily from writing $T_{1,j}=T_{j-1}\cdots T_1\cdots T_{j-1}$, the braid-type relations between $U$ and $T_1$, and repeated application of the relations $[U,T_i]=0$ and $T_i^2=1$ for $i\in [2,j-1]$.
\end{proof}

\subsection{The commutative subalgebra $\mathcal P_Y$}
For $i \in [1,n]$ we define elements $Y_i\in H$ by
\begin{equation} \label{def:Y}
Y_i = \Xi_{i,n} + \Xi_{0,i} + t \sum_{j=i+1}^n T_{i,j}.
\end{equation}
From this definition it is clear that
\[
Y_n=\Xi_{n,n}+\Xi_{0,n} = T_n + T_{n-1} \cdots T_1 T_0 T_1 \cdots T_{n-1},
\]
and
\[
Y_{i+1}  = T_i Y_i T_i - t T_i.
\]
We denote by $\mathcal P_Y\subset H \subset \mathcal H$ the subalgebra generated by $Y_1,\ldots,Y_n$. From the definitions of $Y_i$ and the algebra $H =H(\mathbf t_{\mathcal R})$ it follows that $H$ is generated as an algebra by $Y_i$, $T_i$, $i =1,\ldots, n$. Observe that $\sigma(X_i)=-Y_i^\sigma$, so using the duality isomorphism $\sigma$ it is easy to find properties for the $Y_i$'s from the properties of the $X_i$'s. Note that we also have $\psi(X_i)=-Y_i^\sigma$, where $\psi$ is the duality anti-isomorphism.
\begin{prop}
$\mathcal P_Y=\mathbb C[Y_1,\ldots,Y_n]$
\end{prop}
\begin{proof}
This follows from $\mathcal P_Y = \sigma(\mathcal P_X^\sigma)$, $\mathcal P_X^\sigma=\mathbb C[X_1^\sigma,\ldots,X_n^\sigma]$, and the fact that $\sigma$ is an isomorphism.
\end{proof}

\begin{prop} \label{prop:crossrel H}
We have the following relations in $H$,
\[
T_i p(Y)-(s_ip)(Y) T_i = d_i(Y;\mathbf t^\sigma)\big[p(Y)-(s_ip)(Y)\big],\qquad i\in[1,n],
\]
for $p \in \mathcal P$.
\end{prop}
\begin{proof}
This follows from applying the isomorphism $\sigma_\sigma$ to the relations \eqref{eq:relTPX} in $\mathcal H_\sigma$.
\end{proof}
\begin{cor} \label{cor:daha=PHP}
As a vector space, $H$ is isomorphic to $\mathcal P_Y \otimes H_0$ and $H_0 \otimes \mathcal P_Y$. Consequently,
\[
\mathcal H \simeq \mathcal P_Y \otimes H_0 \otimes \mathcal P_X \simeq \mathcal P_X \otimes H_0 \otimes \mathcal P_Y
\]
as vector spaces.
\end{cor}
\begin{rem}
It should be observed that the elements $Y_i$ are not invertible, which is different from the usual (double) affine Hecke algebra.
\end{rem}

\subsection{The rational GDAHA for type $\widetilde D_4$} \label{ssec:rational GDAHA}
We end this section by showing that the rational DAHA $\mathcal H(\mathbf t)$ is isomorphic to the rational generalized double affine Hecke algebra (GDAHA) of rank $n$ attached to an affine Dynkin diagram of type $\tilde D_4$ as defined in \cite[Definition 2.2.1]{EGO}. We recall the definition here. Actually, in \cite{EGO} a rational GDAHA of rank $n$ is associated to any star-like graph that is not a finite Dynkin diagram.

Let $\nu \in \mathbb C$, $\mu=(\mu_{i_0},\mu_{i_1},\mu_{i_2},\mu_{i_3},\mu_{i_4}) \in \mathbb C^5$, and let $\gamma_k=\gamma_k(\mu)$ ($k \in [1,4]$) be complex numbers such that $\sum_{k=1}^4 \gamma_k=\mu_{i_0}$.
\begin{Def} \label{Def:rat GDAHA}
The algebra $B = B(\mu,\nu)$ is the complex associative algebra generated by $V_{i,k}$ ($i=1,\ldots,n$, $k=1,\ldots,4$) and the symmetric group $\mathcal S_n$ with the following relations: for any $i,j,k \in [1,n]$ with $i \neq j$, and $l,m \in [1,4]$,
\begin{gather*}
s_{ij} V_{i,l} = V_{j,l} s_{ij}, \\
s_{ij} V_{k,l} = V_{k,l} s_{ij} \quad \text{if } k\neq i,j, \\
(V_{k,l}-\gamma_l)(V_{k,l}-\gamma_l-\mu_{i_r})=0, \\
V_{i,1}+V_{i,2}+V_{i,3}+V_{i,4}=\nu \sum_{j \neq i} s_{ij}, \\
[V_{i,l},V_{j,l}]=\nu(V_{i,l} - V_{j,l}) s_{ij}, \\
[V_{i,l},V_{j,m}]=0,\quad  l \neq m,
\end{gather*}
where $s_{ij} \in \mathcal S_n$ denotes the transposition $i \leftrightarrow j$.
\end{Def}
The algebra $B$ essentially only depends on the parameters $\mu$ and $\nu$, see also \cite[Remark 2.2.2]{EGO}. We are going to show that $B$ is isomorphic to $\mathcal H$ for appropriate $\mu$ and $\nu$. For this the following lemma is useful.
\begin{lem} \label{lem:rat GDAHA}
The algebra $B$ is generated by $V_l:=V_{1,l}$ ($l\in [1,4]$) and $s_i:=s_{i,i+1}$ ($i\in[1,n-1])$, with relations, for $l \in [1,4]$, $i \in [1,n-1]$ and $j \in [2,n]$,
\begin{gather*}
s_i^2=1, \\
s_i s_{i+1} s_i = s_{i+1} s_i s_{i+1}, \\
s_{i} V_{l} = V_{l} s_{i}, \quad i\neq 1,\\
(V_{l}-\gamma_l)(V_{l}-\gamma_l-\mu_{i_r})=0, \\
V_{1}+V_{2}+V_{3}+V_{4} = \nu \sum_{k=2}^{n-1} s_{1k}, \\
[V_{l},s_{1j}V_{l}s_{1j}] = \nu[V_{l},s_{1j}],\\
[V_l,s_{1j}V_{m}s_{1j}]=0, \quad m \in [1,4],\ m \neq l,
\end{gather*}
where $s_{1j}=s_1s_2\cdots s_{j-2}s_{j-1} s_{j-2} \cdots s_2 s_1$.
\end{lem}
The proof is a straightforward exercise. We can now show that the rational GDAHA $B$ is isomorphic to the algebra $\mathcal H$.
\begin{thm} \label{thm:B=daha}
Let $\mu(\mathbf t) = (\frac12-t_0-u_0-t_n-u_n, 2t_0,2u_0,2t_n,2u_n)$. The assignments
\begin{gather*}
\varphi(s_{i})=T_i, \quad i \in [1,n-1],\\
\varphi(V_{1})=T_0+\gamma_1+t_0, \qquad \varphi(V_{2}) = T_0^\vee+\gamma_2+u_0,\\
\varphi(V_{3}) = \Xi_{1,n}+\gamma_3+t_n, \qquad \varphi(V_{4}) = \Xi_{1,n}^\vee+\gamma_4+u_n,
\end{gather*}
extend uniquely to an algebra isomorphism $\varphi: B( \mu(\mathbf t),-t) \rightarrow \mathcal H(\mathbf t)$.
\end{thm}

\begin{proof}
Note that $\mathcal H$ is generated as an algebra by $T_0, T_0^\vee, \Xi_{1,n}, \Xi_{1,n}^\vee$ and $T_1, \ldots, T_{n-1}$, hence $\varphi$ maps generators to generators and is therefore surjective. It is an almost straightforward computation to check that $\varphi$ preserves the relations for $B$ from Lemma \ref{lem:rat GDAHA}. Let us remark that the relations
\begin{equation}
[V_l, s_{1j} V_l s_{1j}]= \nu[V_l,s_{1j}], \qquad l \in [1,4], \ j \in [2,n],
\end{equation}
are equivalent to the braid-type relations from Lemma \ref{lem:braidrel}.

We define $\tilde{\varphi}:\mathcal H(\mathbf t) \rightarrow B(\mu(\mathbf t),-t)$ on generators by
\begin{gather*}
\tilde{\varphi}(T_i)=s_i, \quad i\in[1,n-1],\\
\tilde{\varphi}(T_0) = V_1 - \gamma_1-t_0, \qquad \tilde{\varphi}(T_0^\vee) = V_2 - \gamma_2-u_0 \\
\tilde{\varphi}(T_n) = s_{1n}V_3s_{1n} - \gamma_3-t_n \qquad \tilde{\varphi}(T_n^\vee) = s_{1n}V_1s_{1n} - \gamma_4-t_n.
\end{gather*}
It is now easily checked that $\tilde{\varphi}$ preserves the relations for $\mathcal H$ from Theorem \ref{thm:daha=H Hv}, so that $\tilde{\varphi}$, extended to $\mathcal H$ as an algebra homomorphism, is the inverse of $\varphi$.
\end{proof}

\section{The polynomial representation} \label{sec:representation}
In this section we define and study a representation of $\mathcal H$ in terms of difference-reflection operators.\\

Let $\chi:H \rightarrow \mathbb C$ be the trivial representation defined by $\chi(T_i)=\chi_i$ with
\begin{equation} \label{def:chi}
\chi_i =
\begin{cases}
t_i, & i = 0,n,\\
1, & i \in [1,n-1].
\end{cases}
\end{equation}
Using $\mathcal H \simeq H \otimes \mathcal P_X$ as vector spaces, and identifying $\mathcal P_X$ with $\mathcal P$, we define the representation $\pi:\mathcal H \rightarrow \mathrm{End}(\mathcal P)$ to be the induced representation
\[
\pi = \mathrm{Ind}_{H}^{\mathcal H}(\chi).
\]
We call $\pi$ the polynomial representation of $\mathcal H$. From the cross-relations \eqref{eq:relTPX} we find the action of the generators of $\mathcal H$.
\begin{prop}
Let $p \in \mathcal P$. The actions of the generators $T_i$ ($i\in [0,n]$) and $X_i$ ($i\in [1,n]$) on $\mathcal P$ are given by
\begin{align*}
\big(\pi(T_i)p\big)(x) &= c_i(x) (s_ip)(x) - d_i(x) p(x),&& i\in[0,n],\\
\big(\pi(X_i)p\big)(x) &= x_ip(x), && i \in [1,n],
\end{align*}
where $c_i(x)=c_i(x;\mathbf t)$ are rational functions defined by
\[
c_i(x) =
d_i(x)+\chi_i, \qquad i\in[0,n].
\]
\end{prop}
Explicitly, the rational functions $c_i$ are given by
\[
c_i(x) =
\begin{cases}
\dfrac{ \big(t_i+u_i+a_i^\vee(x)\big)\big(t_i-u_i+a_i^\vee(x)\big) }{ a_i(x)}, & i=0,n,\\
\dfrac{ t + a_i(x) }{a_i(x)}, & i \in [1,n-1].
\end{cases}
\]
An easy computation shows that
\begin{equation} \label{eq:c+sc}
c_i(x)+(s_ic_i)(x)=2\chi_i, \qquad i\in [0,n].
\end{equation}
We will also use the notation, for $\alpha \in \mathcal R_r$,
\[
c_\alpha(x;\mathbf t) =
\begin{cases}
\dfrac{ \big((\mathbf t(\alpha)+\mathbf t(\alpha^\vee)+\alpha^\vee(x)\big)\big(\mathbf t(\alpha)-\mathbf t(\alpha^\vee)+\alpha^\vee(x)\big) }{ \alpha(x)}, & \langle\alpha,\alpha\rangle=4,\\
\dfrac{\mathbf t(\alpha) + \alpha(x) }{\alpha(x)}, & \langle \alpha, \alpha \rangle = 2.
\end{cases}
\]
Observe that $c_{a_i}=c_i$ for $i \in [0,n]$, and $wc_\alpha =c_{w\alpha}$ for $w \in W$.
\\
\begin{rem} \label{rem:limits}
(i) Note that for $t=0$ we have $\pi(T_i)=s_i$, $i \in [1,n-1]$. Moreover, setting $(t_0,u_0,t_n,u_n) = (a,-a,a,-a)$ we find
\[
\lim_{a \rightarrow \infty} \big(\pi(t_i^{-1} T_i) p\big)(x) = (s_ip)(x), \qquad i=0,n,
\]
for $p \in \mathcal P$.

(ii) Noumi's \cite{N} representation of the affine Hecke algebra of type $\widetilde C_n$ on $\mathbb C[q^{\pm x_1},\ldots,q^{\pm x_n}]$ is given by
\[
V_i = k_{a_i}+ k_{a_i}^{-1} \frac{ (1-k_{a_i} k_{a_i/2} q^{a_i(x)/2})(1+k_{a_i}k_{a_i/2}^{-1}q^{a_i(x)/2} )}{1-q^{a_i(x)}}(s_i-1), \quad i \in [0,n].
\]
Here $k$ is a multiplicity function on $\mathcal R_{nr}$, and $k_{a_i/2}=1$ if $a_i/2 \not\in \mathcal R_{nr}$. Setting
\[
(k_{a_0},k_{a_0/2},k_{a_1},k_{a_n},k_{a_n/2}) = (-\mathrm i q^{t_0}, \mathrm i q^{u_0}, q^{t/2}, -\mathrm i q^{t_n}, \mathrm i q^{u_n}),
\]
we find formally for $q \rightarrow 1$ (see also \cite[Section 2.4]{G} and \cite[Section 4]{Ch1})
\begin{align*}
&\lim_{q \rightarrow 1}\frac{ 1-\mathrm i V_i}{1-q} = \pi(T_i), && i=0,n,\\
&\lim_{q \rightarrow 1}V_i = \pi(T_i), && i \in [1,n-1].
\end{align*}
So the representation $\pi$ of $\mathcal H$ can formally be obtained as a limit of Noumi's representation of the DAHA associated to $\mathcal R_{nr}$, see \cite{S1}. In view of the next proposition the algebra $\mathcal H$ may be considered as a degeneration of the DAHA associated to $\mathcal R_{nr}$.
\end{rem}

\begin{prop} \label{prop:pi faithful}
The representation $\pi: \mathcal H \rightarrow \mathrm{End}(\mathcal P)$ is faithful.
\end{prop}
\begin{proof}
Suppose that $\pi\big(\sum_{w \in W} f_w(X) T_{\overline{w}}\big)=0$, with $f_w(X) \in \mathcal P_X$. We may write $\pi(T_{\overline w}) = \sum_{u \leq w} a_{w,u}(x)u$ with $a_{w,u}(x) \in \mathbb C(x)$ and $a_{w,w}(x)$ is nonzero. Let $d(x)$ be the product of the denominators of all $a_{w,u}(x)$, then
\[
0= d(x) \sum_{\substack{u,w \in W\\ u \leq w}} f_w(x) a_{w,u}(x) u.
\]
We may consider the expression on the right as an element of the degenerate DAHA  $\mathfrak H(\mathbf 0) \simeq \mathcal P\otimes W$, with $\mathbf 0$ the multiplicity function which is equal to zero on all $W$-orbits in $\mathcal R_r$. Now it follows in the same way as in the proof of Proposition \ref{prop:basis T} that $f_w(x)$ is the polynomial identically equal to zero for all $w \in W$, hence $\pi$ is a faithful representation of $\mathcal H$.
\end{proof}

The representation $\pi$ restricted to the subalgebra $H$ gives a representation of $H$. We are going to decompose $\mathcal P$, considered as a $\pi(H)$-module, into irreducible $\pi(H)$-modules.
In view of Proposition \ref{prop:pi faithful} we identify from here on the algebra $\mathcal H$ with the algebra $\pi(\mathcal H) \subset \mathrm{End}(\mathcal P)$.

It will be useful to understand the action of the operators $T_i$ on monomials. For $\nu=\sum_{i=1}^n \nu_i \epsilon_i \in \mathbb Z_{\geq 0}^n$ we write $|\nu| = \sum_{i=1}^n \nu_i$.
\begin{lem} \label{lem:Tx}
For $\lambda \in \Lambda$ and $\nu=\phi(\lambda)$, we have
\[
\begin{split}
T_{0} x^{\nu} &=
\begin{cases}
\displaystyle \big(t_0+\langle \lambda,\epsilon_1\rangle\big)x^{\nu}+ \sum_{|\mu| < |\nu|}c_{\nu\mu}x^{\mu}, & \langle\lambda,\epsilon_1\rangle\geq 0,\\
\displaystyle x^{\phi(s_{\epsilon_1}\lambda)}-\big(t_0-\langle \lambda,\epsilon_1\rangle \big)x^{\nu} + \sum_{|\mu| < |\nu|}c_{\nu\mu}x^{\mu}, & \langle\lambda,\epsilon_1\rangle<0,
\end{cases}\\
T_i x^\nu &= x^{\phi(s_i \lambda)} + \sum_{|\mu|<|\nu|} c_{\nu\mu} x^\mu, \qquad i \in [1,n-1],\\
T_{n} x^{\nu} & =
\begin{cases}
\displaystyle t_n x^{\nu}, & \langle\lambda, \epsilon_n\rangle\geq 0,\\
\displaystyle - x^{\phi(s_{n}\lambda)} - t_n x^{\nu} +\sum_{|\mu| < |\nu|}c_{\nu\mu}x^{\mu}, & \langle\lambda, \epsilon_n\rangle <0,
\end{cases}
\end{split}
\]
for certain coefficients $c_{\nu\mu} \in \mathbb C$.
\end{lem}
\begin{proof}
For $T_0$ and $T_n$ this is proved in \cite[Proposition 2.5]{G}. For $T_i$, $i \in [1,n-1]$, we write $T_i = s_i+t D_{i}$, where
\[
(D_{i} p)(x) = \frac{(s_i p)(x) -p(x)}{a_i(x)}, \qquad p \in \mathcal P.
\]
We use $s_i(x^\nu)= x^{s_i\nu}$. Then from $s_i \nu = \nu-m a_i$, where $m= \langle\nu, a_i\rangle $, we obtain
\[
\begin{split}
D_i x^\nu =& \frac{ (s_i x)^\nu - x^\nu }{a_i(x)} = \frac{ x^{\nu}(x^{-m a_i}-1)}{x_{i+1}(x^{a_i}-1)}\\
=& \begin{cases}
-x^{\nu-\epsilon_{i+1}-a_i}- x^{\nu-\epsilon_{i+1}-2a_i}- \ldots - x^{\nu-\epsilon_{i+1}-ma_i}, & m>0,\\
0, & m=0,\\
x^{\nu-\epsilon_{i+1}}+ x^{\nu-\epsilon_{i+1}+a_i} + \ldots + x^{\nu-\epsilon_{i+1}-(1+m)a_i}, & m<0.
\end{cases}
\end{split}
\]
All these terms are of degree lower than $|\nu|$.
\end{proof}

\subsection{Intertwiners}
For $i \in [0,n]$ we define the elements $S_i \in \mathcal H$, called intertwiners, by
\[
\begin{split}
S_0 &= [U_n, a_0(Y)],\\
S_i &= [T_i, a_i(Y)], \qquad i \in [1,n].
\end{split}
\]
We define another action of $W$ on $V_{\mathbb C}$ by $s_i\cdot x = s_i x$ for $i \in [1,n]$ and
\[
s_0 \cdot x = (-x_1-1,x_2,\ldots,x_n), \qquad x = (x_1,\ldots, x_2) \in V_{\mathbb C}.
\]
We also define a dot-action of $W$ on $\mathcal P$ by $(w\cdot p)(x) = p(w^{-1} \cdot x)$.
Note in particular that the action of the commutative subalgebra $\tau(\Lambda)$ is given by $\big(\tau(\epsilon_i) \cdot p\big)(x) =  p(x-\epsilon_i)$.
\begin{lem} \label{lem:intertw}
The intertwiners $S_i$ satisfy the following relations in $\mathcal H$:
\begin{enumerate}[(i)]
\item The braid relations of type $\widetilde C_n$
\begin{align*}
S_{i}S_{i+1} S_{i} S_{i+1} &=S_{i+1} S_{i} S_{i+1} S_{i} , && i = 0,n-1,\\
S_i S_{i+1} S_i  &= S_{i+1} S_i S_{i+1}, && i \in[ 1,n-2],\\
S_i S_j& = S_j S_i, && i,j \in [0,n], \ |i-j| \geq 2.
\end{align*}
\item The quadratic relations $S_i^2 = q_i(Y)$ with $q_i(x) = -2a_i(x)^2 c_{a_i}(x) c_{-a_i}(x)$, i.e.,
\[
q_i(x) =
\begin{cases}
4\big((u_n+u_0)^2-(\frac12+x_1)^2\big)\big((u_n-u_0)^2-(\frac12+x_1)^2\big), & i=0,\\
4\big(t^2-a_i(x)^2\big), & i\in [1,n-1],\\
4\big((t_n+t_0)^2-x_n^2\big)\big((t_n-t_0)^2-x_n^2\big), & i=n.
\end{cases}
\]
\item For $p \in \mathcal P$,
\[
S_i p(Y) = (s_i \cdot p)(Y) S_i, \ i \in [0,n].
\]
\end{enumerate}
\end{lem}
\begin{proof}
First we apply the duality isomorphism $\sigma:\mathcal H \rightarrow \mathcal H_\sigma$ and then the representation $\pi_\sigma:\mathcal H_\sigma \rightarrow \mathrm{End}(\mathcal P)$, then by \eqref{eq:relTPX} we have
\[
(\pi_\sigma \circ \sigma)(S_i) = -2a_i(x)c_i(x) s_i.
\]
Properties (i) and (ii) follow from this after an easy calculation. Property (iii) follows in the same way, using also $(\pi_\sigma \circ \sigma)(p(Y)) = p(-x)$.
\end{proof}
For $w\in W$ we define intertwiners $S_w$ by $S_1=1$, and
\[
S_w = S_{i_1} \cdots S_{i_r}
\]
for $w=s_{i_1}\cdots s_{i_r} \in W$ a reduced expression.
This is independent of the choice of the reduced expression, since the $S_i$'s satisfy the braid relations of type $\widetilde C_n$.
From Proposition \ref{lem:intertw} we now find the following relation in $\mathcal H$.
\begin{cor} \label{cor:Swp}
For $p \in \mathcal P$ and $w \in W$ the following relation holds in $\mathcal H$:
\[
S_{w} (w^{-1}\cdot p)(Y) = p(Y) S_{w}.
\]
\end{cor}

\subsection{Nonsymmetric multivariable Wilson polynomials} \label{ssec:nonsym Wilson}
For $\lambda \in \Lambda$ let $\lambda^+$ be the unique element in $W_0 \lambda \cap \Lambda^+$. There is a unique shortest element (with respect to the length-function $l$ on $W$) $v_\lambda \in W_0$ such that $v_\lambda \cdot \lambda=  \lambda^+$. We define $u_\lambda \in W$ by $u_\lambda = \tau(-\lambda)v_{\lambda}^{-1}$. We give a few useful properties of $u_\lambda$, which are proved in \cite[Section 2.4]{M}.
\begin{lem} \label{lem:u lambda}
For $\lambda \in \Lambda$ the element $u_\lambda$ has the following properties:
\begin{enumerate}[(i)]
\item $\tau(-\lambda) = u_\lambda v_\lambda$ and $\tau(-\lambda^+)= v_\lambda u_\lambda$.
\item $u_\lambda$ is the unique shortest element in $W$ satisfying $u_\lambda \cdot 0 = \lambda$.
\item If $a_i(\lambda) \neq 0$ for $i \in [0,n]$, then $s_iu_\lambda = u_{s_i\cdot\lambda}$.
\item If $a_i(\lambda)<0$ for $i \in [0,n]$, then $l(u_{s_i\cdot\lambda}) = l(u_\lambda)-1$.
\item If $a_i(\lambda)=0$ for $i \in [1,n]$, then there exists a $j \in [1,n]$ such that $s_i u_\lambda = u_\lambda s_j$.
\end{enumerate}
\end{lem}
We now define polynomials labeled by $\lambda \in \Lambda$ generated by the intertwiners from the constant polynomial $1 \in \mathcal P$, see also \cite{NUKW}.
\begin{Def} \label{def:nonsymm Wilson pol}
For $\lambda \in \Lambda$ we define the nonsymmetric multivariable Wilson polynomial $p_\lambda \in \mathcal P$ by
\[
p_\lambda = S_{u_\lambda} 1.
\]
\end{Def}
We see that $p_\lambda$ is a nonzero polynomial for any $\lambda \in \Lambda$. We show that the polynomials $p_\lambda$ are eigenfunctions of the $Y$-operators.

\begin{rem}
We show later on that the polynomials $p_\lambda$ are nonsymmetric versions of the multivariable Wilson polynomials as defined by Van Diejen \cite{D1}, which justifies the name. In this paper we will often call $p_\lambda$ a Wilson polynomial.
\end{rem}

Note that the constant function $p_0=1 \in \mathcal P$ satisfies $T_i p_0 = \chi_i p_0$ for $i \in [0,n]$, see \eqref{def:chi}. Now from the definition of the $Y$-operators we see that
\[
Y_i p_0 = \gamma_{0,i} p_0, \qquad \gamma_{0,i} =  t_0+t_n + (n-i)t.
\]
For $\lambda \in \Lambda$ we define elements $\gamma_\lambda \in V_{\mathbb C}$ as follows:
\[
\gamma_{\lambda} = u_\lambda\cdot \gamma_0, \qquad \gamma_0 = \sum_{i=1}^n \gamma_{0,i} \epsilon_i.
\]
From the definition of $u_\lambda$ it follows that $\gamma_\lambda$ can also be written as $\lambda+v_\lambda^{-1} \cdot \gamma_0$. From here on we assume that the parameters $t_0,t_n,t$ are such that $\gamma_\lambda \neq \gamma_\mu$ if $\mu \neq \lambda$.
We have the following useful lemma.
\begin{lem} \label{lem:sp}
Let $p \in \mathcal P$ and $\lambda \in \Lambda$, then
\begin{enumerate}[(i)]
\item $s_i\cdot \gamma_\lambda = \gamma_{s_i\cdot \lambda}$ for $i \in [0,n]$.
\item $(s_ip)(-\gamma_\lambda) = p(-\gamma_{s_i\cdot\lambda})$ and $(s_i\cdot p)(\gamma_\lambda) = p(\gamma_{s_i \cdot \lambda})$ for $i \in [0,n]$ with $s_i\cdot\lambda \neq \lambda$.
\end{enumerate}
\end{lem}
For the proof see \cite[Theorem 5.3]{S1}.

\begin{thm} \label{thm:Yp=gamma p}
For $f \in \mathcal P$ and $\lambda \in \Lambda$, we have
\[
f(Y)p_\lambda = f(\gamma_\lambda)p_\lambda.
\]
\end{thm}
\begin{proof}
By Corollary \ref{cor:Swp} and Lemma \ref{lem:sp} we have
\[
f(Y) p_\lambda = f(Y) S_{u_\lambda} 1 = S_{u_\lambda} (u_\lambda^{-1}\cdot f)(Y)1 = (u_\lambda^{-1}\cdot f)(\gamma_0) S_{u_\lambda}1 = f(\gamma_{u_\lambda\cdot0}) p_\lambda = f(\gamma_\lambda)p_\lambda. \qedhere
\]
\end{proof}
For $\lambda \in \Lambda$ we define eigenspaces
\[
\mathcal P_{\lambda} = \{ p \in \mathcal P \ |\ f(Y) p = f(\gamma_\lambda) p \text{ for all } f \in \mathcal P \}.
\]
Clearly, we have $p_\lambda \in \mathcal P_\lambda$, so $\mathcal P_\lambda$ is nonempty and nonzero for every $\lambda \in \Lambda$.
\begin{thm} \label{thm:basis}
The set $\{p_\lambda \ |\ \lambda \in \Lambda \}$ is a basis for $\mathcal P$.
\end{thm}
\begin{proof}
We first show that the eigenspaces $\mathcal P_\lambda$ are all 1-dimensional. Let $p \in \mathcal P_\lambda$. Observe that for $t=0$ we have $T_i = s_i$, which gives us for the $Y$-operators
\[
Y_i = s_i \cdots s_{n-1} T_n s_{n-1} \cdots s_i + s_{i-1} \cdots s_1 T_0 s_1 \cdots s_{i-1}.
\]
This is the $Y$-operator for the rank one version of the algebra $\mathcal H$ (see \cite{G}), acting only on the $i$th variable. For any $m \in \mathbb Z_{\geq 0}$ the rank one $Y$-operator has exactly one (up to a multiplicative constant) polynomial eigenfunction of degree $m$, which is uniquely determined by the coefficient of $x^m$. This means that for $t=0$ we have
\[
p(x) = c\, p_{\phi(\lambda_1)}(x_1) p_{\phi(\lambda_2)}(x_2) \cdots p_{\phi(\lambda_n)}(x_n),
\]
where $p_m$ denotes a one-variable nonsymmetric Wilson polynomial of degree $m$, and $c$ is a nonzero constant. So the coefficient $c_\lambda$ of $x^{\phi(\lambda)}$ in the expansion $p(x) = \sum_{\mu \in \Lambda} c_{\mu} x^{\phi(\mu)}$ is nonzero after setting $t=0$, which implies that $c_\lambda$ is generically nonzero. Now if $\dim \mathcal P_\lambda >1$, then there would be a nonzero polynomial in $\mathcal P_\lambda$ for which the coefficient of $x^{\phi(\lambda)}$ is equal to zero, hence $\dim \mathcal P_\lambda =1$.

Finally, we define for $m \in \mathbb Z_{\geq 0}$,
\[
\mathcal P_{(m)} = \mathrm{span}\Big\{ x^{\phi(\lambda)} \ | \ \lambda=\sum_{i=1}^n \lambda_i \epsilon_i \in \Lambda, \ \sum_{i=1}^n |\lambda_i| \leq m \Big\},
\]
then $\mathcal P = \bigcup_{m=0}^\infty \mathcal P_{(m)}$. Since $\sum_i \lambda_i^+ = \sum_i |\lambda_i|$ for all $\lambda \in \Lambda$, we see from Lemma \ref{lem:Tx} that the intertwiners $S_i$, $i \in [1,n]$, satisfy $S_i \mathcal P_{(m)} \subset \mathcal P_{(m)}$, while $S_0 \mathcal P_{(m)} \subset \mathcal P_{(m+1)}$. It follows that $p_\lambda \in \mathcal P_{(N)}$ where $N$ is the number of times $s_0$ occurs in a reduced expression for $u_\lambda$. By Lemma \ref{lem:u lambda} we have $u_\lambda = v_{\lambda}^{-1}\tau(-\lambda^+)$ with $v_\lambda \in W_0$, so $N$ is also the number of times $s_0$ occurs in $\tau(-\lambda_+)$, i.e., $N=\sum_{i} \lambda_i^+$.
So for any $m \in \mathbb Z_{\geq 0}$ the set $\{ p_{\lambda} \ | \ \sum_i |\lambda_i| \leq m\}\subset \mathcal P_{(m)}$ has the same cardinality as the basis $\{ x^{\phi(\lambda)} \ | \sum_{i} |\lambda_i| \leq m \}$ of $\mathcal P_{(m)}$. Since the polynomials $p_\lambda$, $\sum_i |\lambda_i| \leq m$, are eigenfunctions of the $Y$-operators for pairwise different eigenvalues, they are linear independent in $\mathcal P_{(m)}$, hence the set  $\{p_\lambda \ | \ \sum_i |\lambda_i|\leq m \}$ forms a basis for $\mathcal P_{(m)}$.
\end{proof}

We have the following corollary of the proof of Theorem \ref{thm:basis}.
\begin{cor}
The eigenspaces $\mathcal P_\lambda$, $\lambda \in \Lambda$, are all one-dimensional. Furthermore, a polynomial in $\mathcal P_\lambda$ is uniquely determined by the coefficient of $x^{\phi(\lambda)}$.
\end{cor}

We define for $\lambda \in \Lambda^+$,
\[
\mathcal P(\lambda) = \mathrm{span}\{ p_\mu \ | \ \mu \in W_0 \lambda \}.
\]
We are going to show that $\mathcal P(\lambda)$ is an irreducible $\pi(H)$-module.
\begin{lem} \label{lem:Sp}
Let $\lambda \in \Lambda$ and $i \in [0,n]$, then
\[
S_i p_\lambda =
\begin{cases}
0, & a_i(\lambda)=0,\\
p_{s_i \cdot \lambda}, & a_i(\lambda) > 0,\\
q_i(\gamma_{s_i\cdot\lambda}) p_{s_i\cdot \lambda}, & a_i(\lambda) < 0.
\end{cases}
\]
\end{lem}
\begin{proof}
Let us first assume that $a_i(\lambda) \neq 0$, i.e.~$s_i\cdot\lambda\neq \lambda$. We have $S_i p_\lambda = S_i S_{u_\lambda} 1$. Let $s_{i_1}\cdots s_{i_r}$ be a reduced expression for $u_\lambda$. We need to find out if $s_i s_{i_1} \cdots s_{i_r}$ is a reduced expression for $s_i u_\lambda$. By Lemma \ref{lem:u lambda} we have $u_{s_i\cdot\lambda}=s_iu_\lambda$.  If $l(u_{s_i\cdot\lambda}) > l(u_\lambda)$, then $s_i s_{i_1} \cdots s_{i_r}$ is a reduced expression for $u_{s_i \cdot \lambda}$. In this case $S_i S_{u_\lambda} = S_{u_{s_i\cdot \lambda}}$, which implies $S_i p_\lambda = p_{s_i\cdot\lambda}$. If $l(u_{s_i\cdot\lambda}) < l(u_\lambda)$, then there exists a reduced expression for $u_\lambda$ which starts with $s_i$. This shows that $S_i S_{u_\lambda} = S_i^2 S_{u_{s_i\cdot\lambda}}$. From Lemma \ref{lem:intertw} and Theorem \ref{thm:Yp=gamma p} we then obtain $S_i p_\lambda = S_i^2 p_{s_i\cdot \lambda} = q_i(\gamma_{s_i\cdot \lambda}) p_{s_i\cdot \lambda}$. By Lemma \ref{lem:u lambda} $a_i(\lambda) < 0$ implies $l(u_{s_i\cdot\lambda})< l(u_\lambda)$. Replacing $\lambda$ by $s_i\cdot \lambda$ then shows that $a_i(\lambda)>0$ implies $l(u_{s_i\cdot\lambda})> l(u_\lambda)$. This proves the lemma in case $a_i(\lambda) \neq 0$.

Now suppose that $a_i(\lambda)=0$, i.e.~$s_i\cdot \lambda = \lambda$, which only happens if $i \neq 0$. In this case we have $(s_i u_\lambda)\cdot 0 = s_i\cdot \lambda = \lambda$. Since $u_\lambda$ is the shortest element in $W$ satisfying $u_\lambda \cdot 0 = \lambda$, we have $l(s_iu_\lambda)>l(u_\lambda)$. Furthermore, by Lemma \ref{lem:u lambda} there exists a $j \in [1,n]$ such that $s_i u_\lambda = u_\lambda s_j$. This shows that $S_ip_\lambda = S_i S_{u_\lambda}1 = S_{u_\lambda}S_j 1$. Observe that $S_j 1 = [T_j, a_j(Y)]1=0$, since $T_j 1 =\chi_j$, hence $S_ip_\lambda = 0$.
\end{proof}

The following result assures us that $T_i \mathcal P(\lambda) \subset \mathcal P(\lambda)$ for $i\in [1,n]$. Clearly, we also have $Y_i\mathcal P(\lambda) \subset \mathcal P(\lambda)$ for $i \in [1,n]$. Since the $Y$-operators are diagonalized by $p_\mu$, a nonempty invariant subspace $\widetilde{\mathcal P} \subset \mathcal P(\lambda)$ contains an element $p_\mu$ for some $\mu \in W_0\lambda$. Repeated application of the intertwiners $S_i$, $i \in [1,n]$, then shows that for every $\nu \in W_0\lambda$ we have $p_\nu \in \widetilde{\mathcal P}$, hence $\mathcal P(\lambda)$ is irreducible.
\begin{lem} \label{lem:Tp}
Let $\lambda \in \Lambda$ and $i \in [1,n]$. If $a_i(\lambda)\neq 0$, then
\[
 T_i p_\lambda = A_{i,\lambda} p_\lambda + B_{i,\lambda} p_{s_i\cdot\lambda},
\]
with $A_{i,\lambda} = d_i^\sigma(\gamma_\lambda)$ and
\[
B_{i,\lambda} =
\begin{cases}
\dfrac{1}{2a_i(\gamma_\lambda)}, & a_i(\lambda)>0,\\ \\
\dfrac{q_i(\gamma_{s_i\cdot\lambda})}{2a_i(\gamma_\lambda)}, & a_i(\lambda)<0.
\end{cases}
\]
If $a_i(\lambda)=0$, then $T_ip_\lambda = \chi_i p_\lambda$.
\end{lem}
\begin{proof}
Let $i \in [1,n]$ and $\lambda \in \Lambda$. By Theorem \ref{thm:basis} $T_ip_\lambda$ can be expanded in terms of $p_\mu$, $\mu\in \Lambda$. From the cross-relation in Proposition \ref{prop:crossrel H} we find for any $f \in \mathcal P$,
\begin{equation} \label{eq:Tp}
\big( f(Y)- (s_if)(\gamma_{\lambda})\big) T_i p_\lambda = \big(f(\gamma_\lambda)-(s_if)(\gamma_{\lambda})\big) d_i^\sigma(\gamma_\lambda) p_\lambda.
\end{equation}

Suppose that $a_i(\lambda)\neq 0$, then $s_i\cdot\lambda \neq \lambda$. In this case we find from Theorem \ref{thm:Yp=gamma p} and Lemma \ref{lem:sp} that $T_ip_\lambda$ is indeed of the form $A_{i,\lambda} p_\lambda + B_{i,\lambda} p_{s_i\cdot\lambda}$ with $A_{i,\lambda}=d_i^\sigma(\gamma_\lambda)$. In order to find $B_{i,\lambda}$ we use the intertwiner $S_i$. By Lemma \ref{lem:Sp} we have $S_i p_\lambda = g_{i,\lambda} p_{s_i \cdot \lambda}$ where
\[
g_{i,\lambda} =
\begin{cases}
1, & a_i(\lambda) >0,\\
q_i(\gamma_{s_i\cdot\lambda}), & a_i(\lambda)<0.
\end{cases}
\]
On the other hand, using $S_i = [T_i, a_i(Y)]$ we have
\[
S_i p_\lambda = \big(a_i(\gamma_\lambda)-a_i(\gamma_{s_i\cdot\lambda})\big) B_{i,\lambda} p_{s_i\cdot\lambda},
\]
hence $B_{i,\lambda}=g_{i,\lambda}/2a_i(\gamma_\lambda)$.

Now suppose that $a_i(\lambda)=0$, then Theorem \ref{thm:Yp=gamma p} and \eqref{eq:Tp} imply that $T_ip_\lambda = d_i^\sigma(\gamma_\lambda)p_\lambda$. Since $s_i\cdot\lambda = \lambda$ it follows from the definition of $\gamma_\lambda$ that $a_i(\gamma_\lambda) = t$ for $i \in [1,n-1]$ and $a_n(\gamma_\lambda) = 2t_0+2t_n$. From the explicit expression for $d_i^\sigma(x)$ we then find  $d_i^\sigma(\gamma_\lambda) = \chi_i$.
\end{proof}

From Lemma \ref{lem:Tp} and Theorem \ref{thm:Yp=gamma p} we obtain the following result.
\begin{prop}
The center of $H$ is equal to $\mathcal P_Y^{W_0}$.
\end{prop}
\begin{proof}
From Proposition \ref{prop:crossrel H} it follows directly that the subalgebra $\mathcal P_Y^{W_0}$ commutes with generators $T_i$ of $H$, hence $\mathcal P_Y^{W_0}$ is contained in the center of the algebra $H$.

Let $Z$ be an element in the center of $H$. By Corollary \ref{cor:daha=PHP} we may write $Z = \sum_{w \in W_0}  T_{\overline{w}}f_w(Y)$ with $f_w(Y) \in \mathcal P_Y$. Since $Z$ commutes with $\mathcal P_Y$, $Z$ acts as a constant on $p_\lambda$, hence
\[
\sum_{w \in W_0} f_w(\gamma_\lambda) T_{\overline{w}}p_\lambda = cp_\lambda ,
\]
for some nonzero constant $c$. By Lemma \ref{lem:Tp} we have $T_{\overline w} p_\lambda = \sum_{v \leq w} c^\lambda_{wv} p_{v\cdot\lambda}$ for certain coefficients $c_{wv}^\lambda \in \mathbb C$, and $c_{ww}^\lambda\neq 0$ for infinitely many $\lambda \in \Lambda$. This gives
\[
\sum_{\substack{v,w \in W_0\\v \leq w}} f_w(\gamma_\lambda) c_{wv}^\lambda p_{v\cdot\lambda} = cp_\lambda,
\]
which implies that for $v \neq 1$
\[
\sum_{\substack{w \in W_0\\w \geq v}} f_w(\gamma_\lambda) c_{wv}^\lambda = 0.
\]
Let $u\neq 1$ be a maximal element such that $f_{u}\neq 0$, then $c_{uu}^\lambda=0$ for infinitely many $\lambda \in \Lambda$. We conclude that only $f_1$ is nonzero. Being a central element, $Z=f_1(Y)$ commutes with every intertwiner $S_i$, $i \in [1,n]$. Lemma \ref{lem:intertw} then implies that $s_if_1=f_1$, hence $f_1(Y) \in \mathcal P_Y^{W_0}$.
\end{proof}

\begin{thm} \label{thm:decomp}
The decomposition
\begin{equation} \label{eq:decomp}
\mathcal P = \bigoplus_{\lambda \in \Lambda^+} \mathcal P(\lambda)
\end{equation}
is the multiplicity-free, irreducible decomposition of $\mathcal P$ as a $\pi(H)$-module. Moreover, \eqref{eq:decomp} gives the decomposition of $\mathcal P$ into isotypical components under the action of the center $\mathcal P_Y^{W_0}$ of $H$, and the central character is given by $\chi_\lambda\big(p(Y)\big)=p(\gamma_\lambda)$ for $p(Y) \in \mathcal P_Y^{W_0}$.
\end{thm}

\section{Nonsymmetric multivariable Wilson polynomials} \label{sec:nonsymm Wilson}
In this section we derive a few important properties for the nonsymmetric multivariable Wilson polynomials, such as orthogonality relations, evaluation formulas and the duality property. These properties are obtained in the same way as in the case of the Koornwinder polynomials, see \cite{S1}, \cite{S2}, \cite{St}, \cite{N}.

\subsection{Duality}
We write $x_\lambda = \gamma_{\lambda}(\mathbf t^\sigma)$, $\lambda \in \Lambda$, for the spectrum of the operators $Y^\sigma \in \mathcal H_\sigma$. We define evaluation mappings $\mathrm{Ev}:\mathcal H \rightarrow \mathbb C$ and $\widetilde {\mathrm{Ev}}:\mathcal H_\sigma \rightarrow \mathbb C$ by
\begin{align*}
\mathrm{Ev}(Z) &= \big(Z(1)\big)(-x_0), && Z \in \mathcal H,\\
\widetilde{\mathrm{Ev}}(\widetilde Z) &= \big(\widetilde Z(1)\big)(-\gamma_0), && \widetilde Z \in \mathcal H_\sigma.
\end{align*}
The two evaluation mapping are related via the duality anti-isomorphism $\psi:\mathcal H \rightarrow \mathcal H_\sigma$ from Proposition \ref{prop:duality iso} by
\begin{equation} \label{eq:Ev=Ev}
\mathrm{Ev}(Z)=\widetilde{\mathrm{Ev}}\big(\psi(Z)\big), \qquad Z \in \mathcal H.
\end{equation}
Indeed, for $Z = f(X)T_wg(Y)$ with $f,g \in \mathcal P$ and $w=s_{i_1}\cdots s_{i_r} \in W_0$ a reduced expression, we have, using $\psi(X_i)=-Y_i^\sigma$ and $\psi(T_i)=T_i^\sigma$ for $i \in [1,n]$,
\[
\begin{split}
\widetilde{\mathrm{Ev}}\big(\psi(Z)\big) &= \Big(g(-X^\sigma)T_{w^{-1}}^\sigma f(-Y^\sigma) (1)\Big)(-\gamma_0) \\
&=\chi_\sigma(T_{w^{-1}}^\sigma)f(-x_0)g(\gamma_0) \\
&=\Big(f(X)T_{w} g(Y) (1)\Big)(-x_0)\\
&= \mathrm{Ev}(Z).
\end{split}
\]
Here we use the reduced expression $w^{-1} = s_{i_r} \cdots s_{i_1}$, and $\chi$ is the trivial representation of $H_0$.
Note that $\chi(T_w)=\chi_\sigma(T_{w^{-1}})$ for any word $w \in W_0$ (and $\chi_\sigma:H_0^\sigma\rightarrow \mathbb C$ is defined in the same way as $\chi:H_0\rightarrow \mathbb C$). By the PBW-property for $\mathcal H$ from Corollary \ref{cor:daha=PHP} we then see that \eqref{eq:Ev=Ev} holds for all $Z \in \mathcal H$.

With the evaluation mappings and the duality anti-isomorphism $\psi:\mathcal H \rightarrow \mathcal H_\sigma$ we construct two pairings $B:\mathcal H \times \mathcal H_\sigma \rightarrow \mathbb C$ and $\widetilde B:\mathcal H_\sigma \times \mathcal H \rightarrow \mathbb C$ by
\[
B(Z,\widetilde Z) = \mathrm{Ev}\big(\psi_\sigma(\widetilde Z)Z\big), \qquad \widetilde B(\widetilde Z,Z)= \widetilde{\mathrm{Ev}}\big(\psi(Z)\widetilde Z\big),
\]
where $Z\in \mathcal H$ and $\widetilde Z \in \mathcal H_\sigma$. We have the following properties for these pairings.
\begin{prop} \label{prop:pairing B}
Let $Z,Z_1,Z_2 \in \mathcal H$, $\widetilde Z, \widetilde Z_1, \widetilde Z_2 \in \mathcal H_\sigma$ and $p \in \mathcal P$. Then
\begin{enumerate}[(i)]
\item $B(Z, \widetilde Z) = \widetilde B(\widetilde Z,Z)$;
\item $B(Z_1Z_2,\widetilde Z) = B(Z_2,\psi(Z_1)\widetilde Z)$ and $B(Z,\widetilde Z_1\widetilde Z_2) = B(\psi_\sigma(\widetilde Z_1)Z,\widetilde Z_2)$;
\item $B\big((Zp)(X),\widetilde Z\big) = B\big(Z p(X), \widetilde Z\big)$ and $B\big(Z, (\widetilde Zp)(X^\sigma)\big)= B\big(Z, \widetilde Z p(X^\sigma)\big)$.
\end{enumerate}
Here $p(X)$ is the multiplication operator $\big(p(X)f\big)(x)=p(x)f(x)$ for $f \in \mathcal P$.
\end{prop}
\begin{proof}
Properties (i) and (ii) follow from \eqref{eq:Ev=Ev}. Property (iii) is an immediate consequence of the identity $\big((Zp)(X)\big)(1)=Zp= Z \big(p(X)(1)\big)$ in $\mathcal P$.
\end{proof}
The evaluation $\mathrm{Ev}(p_\lambda(X)) = p_\lambda(-x_0)$ of the nonsymmetric Wilson polynomial is nonzero for generic parameters $\mathbf t=(t_0,u_0,t,t_n,u_n)$. Indeed, for $t=0$, the polynomial $p_\lambda$ becomes a product of one-variable nonsymmetric Wilson polynomials $p_{\phi(\lambda_i)}$ which are all nonzero at $-x_{0,i}=-(t_n+u_n)$, see \cite[Proposition 3.14]{G}. In the next subsection we determine an explicit expression for $p_\lambda(-x_0)$.
\begin{Def}
For $\lambda \in \Lambda$ we write $E(x,\gamma_\lambda;\mathbf t) = E(x,\gamma_\lambda)$ for the constant multiple of the nonsymmetric Wilson polynomial $p_\lambda(x;\mathbf t)$ that takes the value $1$ at $x=-x_0$.
\end{Def}
As a consequence of Proposition \ref{prop:pairing B} we obtain
\[
B\big(p(X^\sigma), E(X,\gamma_\lambda)\big) = B\big(1,(p(-Y)E(\cdot, \gamma_\lambda))(X)\big) = p(-\gamma_\lambda) B\big(1,E(X,\gamma_\lambda)\big),
\]
for any $p \in \mathcal P$ and $\lambda \in \Lambda$. Note that $B(1,E(X,\gamma_\lambda)=\mathrm{Ev}(E(X,\gamma_\lambda))=1$, so we have
\begin{equation} \label{eq:p=B(p,E)}
f(-\gamma_\lambda) = B\big(f(X^\sigma),E(X,\gamma_\lambda)\big), \qquad g(-x_\mu)=\widetilde B\big(g(X),E_\sigma(X^\sigma,x_\mu)\big),
\end{equation}
for $f,g\in \mathcal P$ and $\lambda,\mu \in \Lambda$. The second identity is derived in the same way as the first. This immediately leads to the duality property for the renormalized nonsymmetric Wilson polynomials.
\begin{thm} \label{thm:duality prop}
For $\lambda, \mu \in \Lambda$ the renormalized nonsymmetric Wilson polynomials satisfy the duality property
\[
E(-x_\mu,\gamma_\lambda;\mathbf t) = E(-\gamma_\lambda,x_\mu;\mathbf t^\sigma).
\]
\end{thm}
\begin{proof}
This follows from Proposition \ref{prop:pairing B}(i) if we set $f=E_\sigma(\cdot,x_\mu)$ and $g=E(\cdot,\gamma_\lambda)$ in \eqref{eq:p=B(p,E)}.
\end{proof}

As a consequence of the duality property we obtain that the actions of $T_i$, $i \in [1,n]$, and $U_n$ on the renormalized nonsymmetric Wilson polynomials can be written as difference-reflection operators acting on the spectral parameters.
\begin{prop} \label{prop:spectral op}
Let $\lambda \in \Lambda$, then
\[
\begin{split}
\big(U_n E(\cdot,\gamma_\lambda)\big)(x) &= c_0^\sigma(-\gamma_\lambda)E(x,\gamma_{s_0\cdot\lambda}) -  d_0^\sigma(-\gamma_\lambda) E(x,\gamma_\lambda), \\
\big(T_i E(\cdot,\gamma_\lambda)\big)(x) &= c_i^\sigma(-\gamma_\lambda)E(x,\gamma_{s_i\cdot\lambda}) -  d_i^\sigma(-\gamma_\lambda) E(x,\gamma_\lambda), \qquad i \in [1,n].
\end{split}
\]
\end{prop}
\begin{proof}
It is enough to prove the result for $x=-x_\mu$ for all $\mu \in \Lambda$. By \eqref{eq:p=B(p,E)} and Proposition \ref{prop:pairing B} we have for $\lambda,\mu \in \Lambda$
\[
\widetilde B\big(E(X,\gamma_\lambda),(T_i^\sigma E_\sigma(\cdot,x_\mu))(X^\sigma)\big) =
\begin{cases}
\big(U_n E(\cdot,\gamma_\lambda)\big)(-x_\mu), & i =0,\\
\big(T_i E(\cdot,\gamma_\lambda)\big)(-x_\mu), & i \in [1,n].
\end{cases}
\]
Writing out $T_i^\sigma E_\sigma(\cdot,x_\mu)$ we find using Proposition \ref{prop:pairing B}
\[
\begin{split}
\widetilde B\big(E(X,\gamma_\lambda),&(T_i^\sigma E_\sigma(\cdot,x_\mu))(X^\sigma)\big)\\
&= \widetilde B\big(E(X,\gamma_\lambda),c_i^\sigma(X^\sigma)(s_iE_\sigma(\cdot,x_\mu))(X^\sigma)-d_i^\sigma(X^\sigma) E_\sigma(X^\sigma,x_\mu)\big)\\
&=  B\big(c_i^\sigma(X^\sigma)(s_iE_\sigma(\cdot,x_\mu))(X)-d_i^\sigma(X^\sigma) E_\sigma(X^\sigma,x_\mu), E(X,\gamma_\lambda)\big) \\
&= c_i^\sigma(-\gamma_\lambda)(s_iE_\sigma(\cdot,x_\mu))(-\gamma_\lambda)-d_i^\sigma(-\gamma_\lambda) E_\sigma(-\gamma_\lambda,x_\mu)\\
&= c_i^\sigma(-\gamma_\lambda) E(-x_\mu, \gamma_{s_i\cdot\lambda}) -d_i^\sigma(-\gamma_\lambda)E(-x_\mu,\gamma_\lambda).
\end{split}
\]
The last two lines follow from \eqref{eq:p=B(p,E)}, Lemma \ref{lem:sp} and the duality property for the nonsymmetric Wilson polynomials, Theorem \ref{thm:duality prop}.
\end{proof}
If $a_i(\lambda)\neq 0$ we already know that $S_i E(\cdot,\gamma_\lambda)= c E(\cdot,\gamma_{s_i\cdot\lambda})$  for some nonzero constant $c$. Using Proposition \ref{prop:spectral op} the constant $c$ can be calculated.
\begin{cor} \label{cor:SE}
Let $i \in [0,n]$, and let $\lambda \in \Lambda$ such that $s_i\cdot \lambda \neq \lambda$, then
\[
S_iE(\cdot,\gamma_\lambda) =2a_i(\gamma_\lambda) c_i^\sigma(-\gamma_\lambda)  E(\cdot,\gamma_{s_i\cdot\lambda}).
\]
\end{cor}

\subsection{The evaluation formula}
Using Corollary \ref{cor:SE} we can determine the value of $p_\lambda(-x_0)$.
\begin{thm} \label{thm:eval form}
For $\lambda \in \Lambda$,
\[
p_\lambda(-x_0) = \prod_{\alpha \in \mathcal R^+_r \cap u_{\lambda}^{-1}\mathcal R^-_r} K_{\alpha}(-\gamma_0),
\]
where
\[
K_\alpha(x) = 2\alpha(-x) c_\alpha^\sigma(x).
\]
\end{thm}
\begin{proof}
By the definition of $E(x,\gamma_\lambda)$ we have $p_\lambda = p_\lambda(-x_0) E(\cdot,\gamma_\lambda)$. Recall that $p_\lambda = S_{u_\lambda}1$, and let $u_\lambda = s_{i_1}\cdots s_{i_r}$ be a reduced expression. Using Corollary \ref{cor:SE} we now find
\[
p_\lambda(-x_0)= \prod_{k=1}^r K_{a_{i_k}}(-\gamma_{(s_{i_{k+1}}\cdots s_{i_r})\cdot 0}).
\]
Recall from Lemma \ref{lem:u lambda} that $u_\lambda$ is the unique shortest element in $W$ such that $u_\lambda \cdot 0=\lambda$, so $(s_{i_k}s_{i_{k+1}}\cdots s_{i_r})\cdot 0\neq (s_{i_{k+1}}\cdots s_{i_r})\cdot 0$. By Lemma \ref{lem:sp} we have $K_\alpha(-\gamma_{w\cdot 0})=K_{w^{-1}\alpha}(-\gamma_0)$ for $w \in W$, so we can write $p_\lambda(-x_0)$ as a product over the set $S=\{\beta_1,\ldots,\beta_r\}$ with $\beta_k = s_{i_r} \cdots s_{i_{k+1}} a_{i_k}$ for $k \in [1,r]$. It is well known that the set $S$ is equal to $\mathcal R^+_r \cap u_\lambda^{-1}\mathcal R^-_r$, which proves the result.
\end{proof}

\subsection{Orthogonality relations}
With the multiplicity function $\mathbf t$ we associate the Wilson parameters $a,b,c,d$ given by
\[
(a,b,c,d) = (t_n+u_n, t_n -u_n, t_0 + u_0 +\frac12, t_0-u_0+\frac12).
\]
We assume from here on that $a,b,c,d,t>0$. We define $\Delta(\cdot)= \Delta(\cdot;\mathbf t)$ and $\Delta_+(\cdot) = \Delta_+(\cdot;\mathbf t)$ by
\[
\Delta_+(x;\mathbf t) =  \prod_{1 \leq j < k \leq n} \frac{ \Gamma(t \pm  x_j \pm x_k) }{ \Gamma(\pm x_j \pm x_k)}   \prod_{j=1}^n \frac{  \Gamma(a \pm x_j) \Gamma(b\pm x_j) \Gamma(c \pm x_j) \Gamma(d \pm x_j)}{ \Gamma(\pm 2x_j) },
\]
and
\[
\Delta(x;\mathbf t) = c_+(x;\mathbf t) \Delta_+(x;\mathbf t),
\]
where
\begin{equation} \label{eq:c+}
c_+(x;{\mathbf t}) = \prod_{\alpha \in \Sigma^-} c_\alpha(x;\mathbf t) =  \prod_{1 \leq j<k\leq n}\frac{(t-x_j+x_k)(t-x_j-x_k)}{(x_k+x_j)(x_k-x_j) } \prod_{j=1}^n \frac{ (a-x_j)(b-x_j) }{-2x_j}
\end{equation}
To the weight functions $\Delta$ and $\Delta_+$ we associate two nondegenerate bilinear forms on $\mathcal P$ by
\[
\begin{split}
\langle f,g \rangle_{\mathbf t}&= \int_{\mathcal (\mathrm i \mathbb R)^n} f(x) g(x) \Delta(x;\mathbf t)\, dx, \\
\langle f,g \rangle_{\mathbf t}^+&= \int_{\mathcal (\mathrm i \mathbb R)^n} f(x) g(x) \Delta_+(x;\mathbf t)\, dx,
\end{split}
\]
where $dx=(2\pi \mathrm i)^{-n}  dx_1 \, dx_2 \cdots dx_n$, and $\mathrm i\mathbb R$ has the standard orientation. The constant terms can be given explicitly;
\[
\langle 1,1 \rangle_{\mathbf t}^+ = 2^n n! \prod_{j=1}^n \frac{ \Gamma(tj) \prod_{1\leq k<l\leq 4}\Gamma(v_k+v_l+(j-1)t) }{\Gamma(t) \Gamma(v_1+v_2+v_3+v_4+(n+j-2)t) },
\]
where $(v_1,v_2,v_3,v_4)=(a,b,c,d)$, see \cite{Gu}. From Stirling's formula it follows that the polynomials are integrable on $(\mathrm i \mathbb R)^n$ with respect to both $\Delta$ and $\Delta_+$. Let us remark that the weight $\Delta_+$ is positive on $(\mathrm{i}\mathbb R)^n$ under the current conditions on $\mathbf t$.

Let $\iota:\mathcal H \rightarrow \mathcal H$ be the anti-isomorphism defined by $\iota(T_i)=T_i$, $i \in [0,n]$, and $\iota(X_j) = X_j$, $j \in [1,n]$.
\begin{lem} \label{lem:symm}
Let $f,g \in \mathcal P$ and $Z \in \mathcal H$, then $\langle Zf,g \rangle_{\mathbf t} = \langle f,\iota(Z)g \rangle_{\mathbf t}$.
\end{lem}
\begin{proof}
It suffices to show that the generators $T_i$ and $X_j$ are symmetric with respect to $\langle \cdot,\cdot \rangle_{\mathbf t}$. For the $X_j$'s this is obvious, so we only need to verify it for the $T_i$'s. Writing out $\langle T_if, g \rangle_{\mathbf t}$ we see that the proof of the lemma boils down to proving the following identity:
\begin{equation} \label{eq:lem_id}
\int_{(\mathrm i\mathbb R)^n} (s_if)(x) g(x) c_{i}(x)\Delta(x) dx=\int_{(\mathrm i\mathbb R)^n} f(x)(s_ig)(x) c_{i}(x)\Delta(x) dx.
\end{equation}
Using the definitions of $c_{i}$, $0=1,\ldots,n$, and $\Delta$ one checks that the function $c_{i}(x) \Delta(x)$ is invariant under the action of $s_i$. Now for $i=1,\ldots,n$ the required identity \eqref{eq:lem_id} follows from  substituting $x \mapsto s_ix$ on the left hand side. For $i=0$ we denote $y=(x_2,\ldots,x_n)$ and $\alpha(x)=\alpha(x_1,y)$ for any function $\alpha$ depending on $x_1,\ldots,x_n$. We substitute $x_1 \mapsto 1-y_1$ in the left hand side of \eqref{eq:lem_id}, then this becomes
\[
\int_{y_1 \in 1+\mathrm i\mathbb R} \int_{y \in \mathcal (\mathrm i\mathbb R)^{n-1}} f(y_1,y) g(1-y_1,y) c_{0}(y_1,y)\Delta(y_1,y)\, dy_1\, dy.
\]
The function $y_1 \mapsto c_{0}(y_1,y) \Delta(y_1,y)$ does not have poles inside the strip $\{ 0 \leq \Re(y_1) \leq 1 \}$ due to the conditions on the parameters. Using Stirling's formula it follows that the integral over the line segment $\{y_1=x+iB \ | \ 0\leq x\leq1 \}$ vanishes for $B \rightarrow \pm \infty$. By Cauchy's theorem we may then shift the contour $1+\mathrm i \mathbb R$ to $\mathrm i \mathbb R$ without changing the outcome of the integral. This proves identity \eqref{eq:lem_id} for $i=0$.
\end{proof}
The $Y$-operators satisfy $\iota(Y_i) = Y_i$, $i \in [1,n]$. From the previous lemma we then immediately obtain orthogonality relations for the Wilson polynomials. Combining this with Theorem \ref{thm:basis} we have the following result.
\begin{thm} \label{thm:orthogonality}
The nonsymmetric Wilson polynomials form an orthogonal basis for $\mathcal P$ with respect to $\langle \cdot, \cdot \rangle_{\mathbf t}$, i.e.,
\[
\langle p_\lambda, p_\mu \rangle_{\mathbf t} = 0,\quad \lambda \neq \mu.
\]
\end{thm}
\begin{rem}
There is a fundamental difference with the DAHA of type $(C^\vee,C_n)$. In that algebra the identity $\iota(Y_i)= Y_i$ is not valid for the corresponding $Y$-operators. Therefore, the polynomial eigenfunctions of the $Y$-operators, the nonsymmetric Koornwinder polynomials, satisfy \emph{bi}orthogonality relations, see \cite{S1}.
\end{rem}
Next we compute the diagonal terms $\langle E(\cdot,\gamma_\lambda),E(\cdot,\gamma_\lambda) \rangle_{\mathbf t}$.

\begin{lem} \label{lem:norm}
For $\lambda \in \Lambda$ and $i \in [0,n]$,
\[
\langle E(\cdot,\gamma_{\lambda}), E(\cdot,\gamma_\lambda) \rangle_{\mathbf t} = \frac{c_{-a_i}^\sigma(\gamma_\lambda)}{c_{a_i}^\sigma(\gamma_\lambda)} \langle E(\cdot,\gamma_{s_i\cdot\lambda}), E(\cdot,\gamma_{s_i\cdot\lambda}) \rangle_{\mathbf t}.
\]
\end{lem}
\begin{proof}
By Lemma \ref{lem:symm} and the definition of the intertwiners $S_i$ we have
\[
\langle S_if,g\rangle_{\mathbf t} = -\langle f,S_ig\rangle_{\mathbf t}, \qquad f,g \in \mathcal P.
\]
Since $S_iE(\cdot,\gamma_\lambda) = b_{\lambda,i} E(\cdot,\gamma_{s_i\cdot\lambda})$, with $b_{\lambda,i}$ given in Corollary \ref{cor:SE}, we have
\[
\langle S_iE(\cdot,\gamma_{\lambda}), S_i E(\cdot,\gamma_\lambda) \rangle_{\mathbf t} = b_{\lambda,i}^2 \langle E(\cdot,\gamma_{s_i\cdot\lambda}), E(\cdot,\gamma_{s_i\cdot\lambda}) \rangle_{\mathbf t}.
\]
On the other hand, by Lemma \ref{lem:intertw} we have
\[
\begin{split}
\langle S_iE(\cdot,\gamma_{\lambda}), S_i E(\cdot,\gamma_\lambda) \rangle_{\mathbf t} &= -\langle E(\cdot,\gamma_{\lambda}), S_i^2 E(\cdot,\gamma_\lambda) \rangle_{\mathbf t} \\
&=-q_i(\gamma_\lambda)\langle E(\cdot,\gamma_{\lambda}), E(\cdot,\gamma_{\lambda}) \rangle_{\mathbf t}.
\end{split}
\]
Now the result follows from $-b_{\lambda,i}^2/q_i(\gamma_\lambda)=c_{-a_i}^\sigma(\gamma_\lambda)/c_{a_i}^\sigma(\gamma_\lambda)$.
\end{proof}
For $\lambda \in \Lambda$ we define
\[
N(\gamma_\lambda;\mathbf t) = \frac1{\langle 1,1 \rangle_{\mathbf t}} \prod_{\alpha \in \mathcal R_r^+ \cap u_\lambda \mathcal R_r^-} \frac{ c_{-\alpha}^\sigma(\gamma_\lambda) }{ c_{\alpha}^\sigma(\gamma_\lambda) }
\]
From $E(\cdot,\gamma_0)=1$, Lemma \ref{lem:norm} and arguments similar as in the proof of Theorem \ref{thm:eval form} we see that $N(\gamma_\lambda)^{-1}$ is the `quadratic norm' for $E(\cdot,\gamma_\lambda)$.
\begin{thm} \label{thm:norm}
For $\lambda \in \Lambda$,
\[
\langle E(\cdot,\gamma_\lambda), E(\cdot,\gamma_\lambda) \rangle_{\mathbf t} = \frac{1}{N(\gamma_\lambda)}.
\]
\end{thm}
\subsection{The Fourier transform}
We define
\[
\mathrm{Spec}(-Y)=\{-\gamma_\lambda \ | \ \lambda \in \Lambda \}.
\]
Let $F$ be the vector space of complex-valued functions on $\mathrm{Spec}(-Y)$ with finite support. A linear basis for $F$ is formed by the functions $\delta_\lambda$, $\lambda \in \Lambda$, given by
\[
\delta_\lambda(-\gamma_\mu) =
\begin{cases}
0,& \mu \neq \lambda,\\
1,& \mu = \lambda.
\end{cases}
\]
We define a bilinear form $[\cdot,\cdot]: F \times F \rightarrow \mathbb C$ by
\[
[f,g]_{\mathbf t} = \sum_{\gamma \in \mathrm{Spec}(-Y)} f(\gamma)g(\gamma) N(-\gamma;\mathbf t)
\]
The functions $\delta_\lambda$, $\lambda \in \Lambda$, are orthogonal with respect to this bilinear form.

We define the polynomial Fourier transform $\mathcal F : \mathcal P \rightarrow F$ by
\[
(\mathcal Fp)(\gamma) = \langle p , E(\cdot,-\gamma) \rangle_{\mathbf t}, \qquad p \in \mathcal P,\ \gamma\in \mathrm{Spec(-Y)}.
\]
Since the bilinear form $\langle \cdot,\cdot \rangle$ is nondegenerate, the map $\mathcal F$ is an injective map. Moreover, $\mathcal FE(\cdot;\gamma_\lambda)$ is a multiple of $\delta_\lambda$, hence $\mathcal F$ maps an orthogonal basis of $\mathcal P$ to an orthogonal basis of $F$, which implies that $\mathcal F$ is surjective.
We also  define a linear map $\mathcal G:F \rightarrow \mathcal P$ by
\[
(\mathcal Gg)(x) = [g, E(x,-\,\cdot)]_{\mathbf t}=\sum_{\gamma \in \mathrm{Spec}(-Y)} g(\gamma) E(x,-\gamma)N(-\gamma;\mathbf t), \qquad g \in F, \ x \in \mathbb C^n.
\]

\begin{thm} \label{thm:inverse F}
The map $\mathcal G: F \rightarrow \mathcal P$ is the inverse of the Fourier transform $\mathcal F:\mathcal P \rightarrow F$. Moreover, we have the Plancherel-type formulas
\[
[\mathcal Ff_1, \mathcal Ff_2]_{\mathbf t} = \langle f_1,f_2 \rangle_{\mathbf t}, \qquad \langle \mathcal Gg_1, \mathcal Gg_2 \rangle_{\mathbf t} = [g_1,g_2]_{\mathbf t},
\]
for $f_1,f_2 \in \mathcal P$ and $g_1,g_2 \in F$.
\end{thm}
\begin{proof}
The proof is straightforward using the orthogonality relations for the nonsymmetric Wilson polynomials, see Theorems \ref{thm:orthogonality} and \ref{thm:norm}.
\end{proof}

The affine Weyl group $W$ has an action on $F$ defined by
\[
(wf)(-\gamma_\lambda) = f(-\gamma_{w^{-1}\cdot \lambda}), \qquad w \in W, \ f \in F.
\]
We now define an action of $\mathcal H_\sigma$ on $F$. For $f \in F$ and $Z \in \mathcal H_\sigma$ we set
\begin{equation} \label{eq:action F}
(Z f)(\gamma) = (Z \overline{f})(\gamma).
\end{equation}
Here $\overline{f}$ denotes an arbitrary function $\overline{f}:\mathbb C^n \rightarrow \mathbb C$ such that $\overline{f}(\gamma)=f(\gamma)$ for any $\gamma \in \mathrm{Spec}(-Y)$, and the action of $\mathcal H_\sigma$ on the right hand side is given by the usual difference-reflection operators. We need to verify that \eqref{eq:action F} is well-defined.
\begin{lem}
The action of $\mathcal H_\sigma$ on $F$ defined by \eqref{eq:action F} is independent of the choice of extension $\overline{f}$ of $f$.
\end{lem}
\begin{proof}
It is enough to verify that the lemma is true for generators $T_i^\sigma$, $i \in [0,n]$, and $X_i$, $i \in [1,n]$, of the algebra $\mathcal H_\sigma$. For the $X_i$'s this is obvious.
Let $f \in F$ and let $\overline{f}$ be an extensions of $f$. By Lemma \ref{lem:sp} we have $(s_i\overline{f})(\gamma)=(s_i f)(\gamma)$ for $\gamma \in \mathrm{Spec}(-Y)$ and $i \in [0,n]$. For the $T_i^\sigma$'s this gives us
\[
(T_i^\sigma \overline{f})(\gamma) = \chi_i f(\gamma) + c_i^\sigma(\gamma) \big((s_i f)(\gamma)-f(\gamma)\big),
\]
for any extension $\overline{f}$ of $f$. This is clearly independent of the choice of extension.
\end{proof}

We have the following intertwining properties for the Fourier transforms $\mathcal F$ and $\mathcal G$.
\begin{prop}
The Fourier transforms $\mathcal F:\mathcal P \rightarrow F$ and $\mathcal G:F \rightarrow \mathcal P$ have the following intertwining properties:
\begin{align*}
\mathcal F\circ Z &= \sigma(Z) \circ \mathcal F, && Z \in \mathcal H,\\
\mathcal G \circ Z &= \sigma_\sigma(Z)\circ \mathcal G, && Z \in \mathcal H_\sigma.
\end{align*}
\end{prop}
\begin{proof}
It is enough to check the intertwining property of $\mathcal F$ for the generators $Y_i,T_i, U_n$, $i \in [1,n]$, of the algebra $\mathcal H$. This can be done using Lemma \ref{lem:symm}, Proposition \ref{prop:spectral op} and Proposition \ref{prop:duality iso}.

The intertwining property of $\mathcal G$ follows from the intertwining property of $\mathcal F$ and Theorem \ref{thm:inverse F}. Let $g \in F$, then $g = \mathcal Ff$ for a unique $f \in \mathcal P$, or equivalently $f = \mathcal G g$. Let $Z \in \mathcal H_\sigma$, then $Zg = \mathcal F(\sigma_\sigma(Z) f)$, which gives us
\[
\mathcal G(Zg) = (\mathcal G \circ \mathcal F)(\sigma_\sigma(Z)f)= \sigma_\sigma(Z) f = \sigma_\sigma(Z) \mathcal Gg. \qedhere
\]
\end{proof}

\section{Symmetric multivariable Wilson polynomials} \label{sec:symm Wilson}
In this section we define and study symmetric multivariable Wilson polynomials. They turn out to coincide with the multivariable Wilson polynomials as defined by Van Diejen \cite{D2}. In \cite{D2} it is shown that the symmetric Wilson polynomials in $n$ variables diagonalize a system of $n$ commuting difference operators, the duality property is obtained, and  orthogonality relations are given. We derive these properties here from the representation theory of the rational DAHA $\mathcal H$, and we provide the link with the nonsymmetric theory.

\subsection{Symmetric Wilson polynomials}
For $\lambda \in \Lambda^+$ we define
\[
\mathcal P(\lambda)^{W_0} = \mathcal P(\lambda) \cap \mathcal P^{W_0}.
\]
Observe that it follows from $\mathcal P=\bigoplus_{\lambda \in \Lambda^+} \mathcal P(\lambda)$ that $\mathcal P^{W_0} = \bigoplus_{\lambda \in \Lambda^+} \mathcal P(\lambda)^{W_0}$.
\begin{prop} \label{prop:P(lambda)W}
For $\lambda \in \Lambda^+$ the space $\mathcal P(\lambda)^{W_0}$ is one-dimensional. Furthermore, any $f \in \mathcal P(\lambda)^{W_0}$ is given by
\begin{equation} \label{eq:f=sumE}
f = \sum_{\mu \in W_0\lambda} d_{\lambda,\mu} E(\cdot,\gamma_\mu),
\end{equation}
with
\[
d_{\lambda,\mu} = d_{\lambda}\, c_+^\sigma(-\gamma_\mu),
\]
for some constant $d_\lambda$ independent of $\mu \in W_0\lambda$.
\end{prop}
\begin{proof}
Let $i \in [1,n]$. From Proposition \ref{prop:spectral op} we know that
\[
(T_i-\chi_i) E(\cdot,\gamma_\mu) = c_{a_i}^\sigma(-\gamma_\mu)\big(E(\cdot,\gamma_{s_i\mu})-E(\cdot,\gamma_\mu)\big),
\]
if $s_i\mu \neq \mu$. Applying $(T_i-\chi_i)$ to the expansion \eqref{eq:f=sumE} and using $(T_i-\chi_i)f=0$ we find a recurrence relation for the coefficients $d_{\lambda,\mu}$;
\begin{equation} \label{eq:recurrence d}
d_{\lambda,s_i\mu} = \frac{c_{a_i}^\sigma(-\gamma_\mu) }{ c_{-a_i}^\sigma(-\gamma_\mu)} d_{\lambda,\mu}.
\end{equation}
Here we also used $c_\alpha(\gamma_{s_i \mu}) = (s_i c_\alpha)(\gamma_\mu) = c_{s_i\alpha}(\gamma_\mu)$. Let $s_{i_1}\cdots s_{i_r}$ be a reduced expression for $v_\mu$ (defined in the beginning of subsection \ref{ssec:nonsym Wilson}), and define
\[
\beta_1 = a_{i_1}, \quad \beta_j = s_{i_{j-1}} \cdots s_{i_1} a_{i_j}, \qquad j=2,\ldots,r,
\]
then iterating \eqref{eq:recurrence d} gives
\[
d_{\lambda,\mu} = d_{\lambda,\lambda} \prod_{j=1}^r \frac{ c^\sigma_{-\beta_j}(\gamma_\lambda) }{ c^\sigma_{\beta_j}(\gamma_\lambda) }.
\]
From the definition of $c_+(x)$ we see that $(s_ic_+)(x) = c_+(x) c_{-a_i}(-x) / c_{a_i}(-x)$, and then
\[
d_{\lambda,\mu} = d_{\lambda,\lambda} \frac{ (v_\mu c_+^\sigma)(-\gamma_\lambda) }{ c_+^\sigma(-\gamma_\lambda)} = d_{\lambda,\lambda} \frac{ c_+^\sigma(-\gamma_\mu) }{ c_+^\sigma(-\gamma_\lambda)}.
\]
If we now set $d_{\lambda,\lambda} = d_\lambda c_+^\sigma(-\gamma_\lambda)$, the lemma follows.
\end{proof}
In Section \ref{sec:symmetrizer} we give an operator that maps $\mathcal P(\lambda)$ onto $\mathcal P(\lambda)^{W_0}$.
\begin{Def} \label{def:E+}
For $\lambda \in \Lambda^+$ we define the symmetric multivariable Wilson polynomial $E^+(x,\gamma_\lambda)$ to be the unique polynomial in $\mathcal P(\lambda)^{W_0}$ that takes the value $1$ at $x=x_0$.
\end{Def}
Since $f(x_0)=f(-x_0)$ if $f \in \mathcal P^{W_0}$ and $E(-x_0,\gamma_\lambda)=1$, by Proposition \ref{prop:P(lambda)W} the symmetric Wilson polynomial $E^+(\cdot,\gamma_\lambda)$ is the polynomial in $\mathcal P(\lambda)^{W_0}$ with constant $d_\lambda$ given by
\[
d_{\lambda} = \Bigg( \sum_{\mu \in W_0\lambda} c_+^\sigma(-\gamma_\mu) \Bigg)^{-1}.
\]
This constant can be evaluated with the following result.
\begin{lem} \label{lem:K(x)}
The function $K = \sum_{w \in W_0} wc_+$ is a constant function, and
\[
K(x) = \sum_{\mu \in W_0\lambda} c_+(-x_\mu) = c_+(-x_0),
\]
for any $\lambda \in \Lambda^+$.
\end{lem}
\begin{proof}
We define the polynomial $m$ by
\[
m(x) = \prod_{\alpha \in \Sigma^-} \alpha(x) =  (-2)^n\prod_{1\leq j < k \leq n} (x_k^2-x_j^2) \prod_{j=1}^n x_j.
\]
This is an anti-symmetric polynomial, i.e., $(wm)(x)=(-1)^{l(w)}m(x)$ for any $w \in W_0$, where $l(w)$ denotes the length of $w$. It is well known that any anti-symmetric polynomial is the product of $m$ with a symmetric polynomial.
The product $m(x) K(x) $ is manifestly an anti-symmetric polynomial;
\[
m(x)K(x) =
\sum_{w \in W_0} (-1)^{l(w)} w\Bigg(\prod_{1 \leq j < k \leq n} (t+x_k-x_j)(t-x_k-x_j) \prod_{j=1}^n (a-x_j)(b-x_j) \Bigg).
\]
Expanding both $m(x)$ and $m(x)K(x)$ in monomials $x^\mu$, $\mu \in \mathbb Z_{\geq 0}^n$, we see that any monomial with nonzero coefficient in the expansion of $m(x)K(x)$ also has nonzero coefficient in the expansion of $m(x)$, hence $K(x)$ is a constant.
Now let $\lambda \in \Lambda^+$, then by Lemma \ref{lem:sp}
\[
K(x) = K(-x_\lambda) = \sum_{w \in W_0} (wc_+)(-x_\lambda) = \sum_{\substack{w \in W_0\\ w\lambda \neq \lambda}} c_+(-x_{w\lambda}) + \sum_{\substack{w \in W_0\\ w\lambda = \lambda}} (wc_+)(-x_{\lambda}).
\]
Consider a term $(wc_+)(-x_\lambda)$ with $w \neq 1$ from the second sum. There exists a simple root $a_i$, $i \in [1,n]$, such that $\alpha = w^{-1} a_i \in \Sigma^-$. Then we see that the factor $(wc_\alpha)(-x_\lambda) = c_{a_i}(-x_\lambda)$ equals zero, since $a_i(x_\lambda)=t$ for $i \in [1,n-1]$ and $a_n(x_\lambda) = 2t_n+2u_n$. This shows that
\[
K(x) = \sum_{\mu \in W_0\lambda} c_+(-x_\mu).
\]
In particular, for $\lambda=0$ this gives $K(x) = c_+(-x_0)$.
\end{proof}

\subsection{Orthogonality relations}
Next we show that the symmetric Wilson polynomials are orthogonal with respect to $\langle \cdot,\cdot \rangle_{\mathbf t}^+$. First a useful lemma. Let $V^+ \in W$ be the primitive idempotent
\[
V^+ = \frac{1}{|W_0|}\sum_{w \in W_0} w,
\]
where $|W_0|=2^nn!$ is the cardinality of the finite Weyl group $W_0$. Note that $V^+ f =f$ if $f \in \mathcal P^{W_0}$.
\begin{lem} \label{lem:inprod}
For $f,g \in \mathcal P^{W_0}$ we have
\[
\langle f,g \rangle_{\mathbf t} = \frac{ c_+(-x_0) }{|W_0|} \langle f,g \rangle_{\mathbf t}^+.
\]
\end{lem}
\begin{proof}
First note that $s_i$, $i \in [1,n]$, is symmetric with respect to $\langle \cdot, \cdot \rangle_{\mathbf t}^+$, since $s_i\Delta^+=\Delta^+$, hence so is $V^+$. Let $f,g \in \mathcal P^{W_0}$, then by Lemma \ref{lem:K(x)}
\[
\langle f,g \rangle_{\mathbf t} = \langle V^+f,g \rangle_{\mathbf t} = \langle V^+f, c_+g  \rangle_{\mathbf t}^+ = \langle f, V^+(c_+g) \rangle_{\mathbf t}^+ = \langle f, (V^+c_+)g \rangle_{\mathbf t}^+ = \frac{ c_+(-x_0) }{|W_0|}\langle f,g \rangle_{\mathbf t}^+.
\]
This proves the lemma.
\end{proof}

\begin{thm} \label{thm:basis+}
The symmetric Wilson polynomials $E^+(\cdot,\gamma_\lambda)$, $\lambda \in \Lambda^+$, form an orthogonal basis for $\mathcal P^{W_0}$ with respect to $\langle \cdot,\cdot \rangle_{\mathbf t}^+$.
\end{thm}
\begin{proof}
Since $\mathcal P^{W_0} = \bigoplus_{\lambda \in \Lambda^+} \mathcal P(\lambda)^{W_0}$ it is clear that the symmetric Wilson polynomials form a basis for $\mathcal P^{W_0}$. From the orthogonality relations for the nonsymmetric Wilson polynomials and the definition of the symmetric Wilson polynomials, it follows that
\[
\langle E^+(\cdot,\gamma_\lambda), E^+(\cdot,\gamma_\mu) \rangle_{\mathbf t}=0,\qquad \lambda \neq \mu.
\]
Now the orthogonality with respect $\langle \cdot ,\cdot\rangle_{\mathbf t}^+$ follows from Lemma \ref{lem:inprod}.
\end{proof}

\subsection{Difference equations}
A difference-reflection operator $D \in \mathcal P_Y^{W_0}$ (considered as a subalgebra of $\mathrm{End}(\mathcal P)$) is of the form
\[
D = \sum_{\lambda \in \Lambda, w \in W_0} c_{\lambda,w}(x) \tau(\lambda) w,
\]
with $c_{\lambda,w}(x) \in \mathbb C(x)$.
Restricted to the algebra of $W_0$-invariant polynomials, it becomes a difference operator which we denote by $D_{sym}$;
\[
D_{sym} = \sum_{\lambda \in \Lambda, w \in W_0} c_{\lambda,w}(x) \tau(\lambda).
\]
Moreover, $D_{sym}$ is $W_0$-invariant, i.e., $w \circ D_{sym} \circ w^{-1} = D_{sym}$ for any $w \in W_0$. Together with Theorem \ref{thm:decomp} this leads to the following property of the symmetric Wilson polynomials.
\begin{thm} \label{thm:YE+=gammaE+}
The symmetric Wilson polynomials satisfy the difference equation
\[
f(Y)_{sym} E^+(\cdot,\gamma_\lambda) = f(\gamma_\lambda) E^+(\cdot,\gamma_\lambda), \qquad \lambda \in \Lambda^+,
\]
for any $f \in \mathcal P^{W_0}$.
\end{thm}
In particular, the symmetric Wilson polynomials are eigenfunctions of the following explicit difference operator.
\begin{prop} \label{prop:2nd order difference}
The symmetric Wilson polynomials $E^+(\cdot,\gamma_\lambda)$, $\lambda \in \Lambda^+$, are eigenfunctions of the second order difference operator $L$ defined by
\[
\begin{split}
L &= \sum_{i=1}^n A_i(x) \big(\tau(\epsilon_i)-1\big) + A_i(-x) \big(\tau(-\epsilon_i)-1\big), \\
A_i(x) &= \frac{ (a+x_i)(b+x_i)(c+x_i)(d+x_i) }{2x_i(2x_i+1)} \prod_{j\neq i} \frac{ (t+x_i+x_j)(t+x_i-x_j) }{(x_i+x_j)(x_i-x_j)},
\end{split}
\]
for eigenvalue $\sum_{i=1}^n \lambda_i\big(\lambda_i+a+b+c+d-1+2(n-i)t\big)$.
\end{prop}
\begin{proof}
Let $f_2 \in \mathcal P^{W_0}$ be defined by $f_2(x)=\sum_{i=1}^n x_i^2$. We show that $L = \big(f_2(Y)-f_2(\gamma_0)\big)\big|_{sym}$. Then by Theorem \ref{thm:decomp} the symmetric Wilson polynomials $E^+(\cdot,\gamma_\lambda)$ are eigenfunctions of $L$ for eigenvalue $\sum_i (\gamma_{\lambda,i}^2- \gamma_{0,i}^2) = \sum_i \lambda_i(\lambda_i + 2 \gamma_{\lambda,i})$, which is precisely the eigenvalue in the proposition.

Recall the reduced expression \eqref{eq:tau} for $\tau(\epsilon_i)$, and recall that $w\tau(\epsilon_i)w^{-1}=\tau(w\epsilon_i)$ for $w \in W_0$. From the explicit expression for the $Y$-operators we see that $Y_i^2$ expressed in terms of the generators $T_j$ contains only terms in which $T_0$ occurs at most once (in the terms in which $T_0$ occurs twice, the $T_0$'s vanish using the quadratic relations in $\mathcal H$). Hence $L$ is of the form
\[
C(x) + \sum_{i=1}^n A_i(x) \tau(\epsilon_i) + B_i(x) \tau(-\epsilon_i)
\]
with $A_i(x), B_i(x), C(x) \in \mathbb C(x)$. Since $Y_i 1 = \gamma_{0,i}$ and therefore $L1=0$, we see immediately that $C(x) = -\sum_i (A_i(x) + B_i(x))$. By the $W_0$-invariance of $L$ the coefficients $A_i$ and $B_i$ must satisfy
\[
s_{\epsilon_{i}-\epsilon_j} A_i(x) = A_j(x), \qquad s_{2\epsilon_{i}}A_i(x) = B_i(x).
\]
Now it is enough to find the coefficient $A_1(x)$ of $\tau(\epsilon_1)$. We write
\begin{equation} \label{eq:sumY}
\sum_{i=1}^n Y_i^2 = \sum_{w \in W} c_w(x) w,
\end{equation}
with $c_w(x) \in \mathbb C(x)$ and every $w \in W$ occurring in this expansion is an ordered subword of
\[
s_i s_{i+1} \cdots s_n s_{n-1} \cdots s_0 s_1 \cdots s_{i-1} \quad \mathrm{or} \quad s_{i-1} s_{i-2} \cdots s_0 s_1 \cdots s_n s_{n-1} \cdots s_{i},
\]
for $i \in [1,n]$. If we write here $s_0 = s_1 \cdots s_n \cdots s_1 \tau(\epsilon_1)$, and we use $u\tau(\epsilon_i)=\tau(u\epsilon_i)u$, $u \in W_0$, to write every $w$ in \eqref{eq:sumY} as $\tau(\pm \epsilon_i) v$ for some $v \in W_0$, then we see that $\tau(\epsilon_1)$ occurs only once. The only contribution to $\tau(\epsilon_1)$ comes from the expression $T_1 \cdots T_n \cdots T_0$ in $Y_1^2$, and this gives us
\[
A_1(x) = c_{\epsilon_1-\epsilon_2}(x) \cdots c_{\epsilon_1-\epsilon_n}(x) c_{2\epsilon_1}(x) c_{\epsilon_1+\epsilon_n}\cdots c_{\epsilon_1+\epsilon_2}(x)c_{\delta+2\epsilon_1}(x).
\]
Now the result follows from writing out this expression.
\end{proof}

\begin{rem}
It can be shown that the difference operator $L$ acts triangularly with respect to the dominance ordering on the $W_0$-symmetric monomial $m_\lambda = \sum_{\mu \in \mathcal S_n\lambda} x^{\phi(\mu)}$, $\lambda \in \Lambda^+$;
\[
L m_\lambda = \gamma_\lambda^+ m_\lambda + \sum_{\mu < \lambda} c_{\lambda\mu} m_\mu,
\]
with $\gamma_\lambda^+$ the eigenvalue from Proposition \ref{prop:2nd order difference}. This shows that the eigenspace of $L$ for eigenvalue $\gamma_\lambda^+$ is one-dimensional. The multivariable Wilson polynomials defined by Van Diejen \cite{D1}, \cite{D2}, are eigenfunctions of $L$ for eigenvalue $\gamma_\lambda^+$, hence they coincide with our symmetric Wilson polynomials.
\end{rem}

\subsection{Duality} Using $f(Y)E^+(\cdot,\gamma_\lambda) = f(\gamma_\lambda) E^+(\cdot,\gamma_\lambda)$ for $f\in \mathcal P^{W_0}$, the duality property for the symmetric Wilson polynomials can be obtained in the same way as for the nonsymmetric ones, see Theorem \ref{thm:duality prop}.
\begin{thm}
For $\lambda,\mu \in \Lambda^+$, we have
\[
 E^+(x_\mu,\gamma_\lambda) = E^+_\sigma(\gamma_\lambda,x_\mu).
\]
\end{thm}
Using the duality property, the difference equations from Theorem \ref{thm:YE+=gammaE+} give rise to recurrence relations for the symmetric Wilson polynomials.

\subsection{The symmetrizer} \label{sec:symmetrizer}
We define an operator $C^+ : \mathbb C(x) \rightarrow \mathbb C(x)$ by
\[
C^+ = V^+ c_+(X) ,
\]
where $c_+(X)$ is multiplication by $c_+(x)$, see \eqref{eq:c+}. We call $C^+$ the symmetrizer.
\begin{lem} \label{lem:symmetrizer}
The symmetrizer $C^+$ has the following properties:
\begin{enumerate}[(i)]
\item For $f \in \mathcal P$ we have $C^+f \in \mathcal P^{W_0}$.
\item For $f \in \mathcal P^{W_0}$ we have $C^+ f =  |W_0|^{-1}c_+(-x_0)f$.
\item For $i \in [1,n]$, $C^+ T_i = \chi_i C^+=T_i C^+$.
\end{enumerate}
\end{lem}
\begin{proof}
(i) Let us denote $|W_0|(C^+f)(x)=K_f(x)$. In the same way as in the proof of Lemma \ref{lem:K(x)} we find that the product $m(x) K_f(x)$ is an anti-symmetric polynomial, from which it follows that $K_f \in \mathcal P^{W_0}$.

(ii) If $f \in \mathcal P^{W_0}$ we see that $K_f = f K_1$. Then the result follows from Lemma \ref{lem:K(x)}.

(iii) Let $i \in [1,n]$. Using $T_i = \chi_i + c_i(X)(s_i-1)$ and $c_{a_i}+c_{-a_i}=2\chi_i$ we find $T_i+\chi_i = (s_i+1)c_{-a_i}(X)$, since $s_i c_{a_i} = c_{s_ia_i}$. Furthermore, from the definition  of $c_+(x)$ it follows that $s_i c_+(X) c_{a_i}(X) = c_+(X) s_i c_{-a_i}(X)$. This gives
\[
\begin{split}
C^+(T_i+\chi_i) &= V^+c_+(X)s_i c_{-a_i}(X) + V^+c_+(X) c_{-a_i}(X)\\
&= V^+s_i c_+(X) c_{a_i}(X) + V^+c_+(X) c_{-a_i}(X) \\
&= V^+ c_+(X) 2\chi_i\\
&=2\chi_i C^+,
\end{split}
\]
since $V^+ s_i = V^+$. This proves the identity $ C^+T_i = \chi_iC^+$. The identity $T_i C^+ = \chi_iC^+$  follows from (i), since $T_if=\chi_if$ for $f \in \mathcal P^{W_0}$.
\end{proof}

We show that the symmetrizer $C^+$ maps $\mathcal P(\lambda)$ onto $\mathcal P(\lambda)^{W_0}$ for all $\lambda \in \Lambda^+$.
\begin{prop} \label{prop:C^+E}
For $\lambda \in \Lambda^+$ and $\mu \in W_0\lambda$, we have
\[
C^+E(\cdot,\gamma_\mu)=\frac{c_+(-x_0)}{|W_0|} E^+(\cdot,\gamma_\lambda).
\]
\end{prop}
\begin{proof}
Let $\lambda \in \Lambda^+$ and $\mu \in W_0\lambda$. We know that $C^+E(\cdot,\gamma_\mu) \in \mathcal P^{W_0}$, so by Theorem \ref{thm:basis+} there exists an expansion
\[
C^+E(\cdot,\gamma_\mu) = \sum_{\nu \in \Lambda^+} b_{\mu,\nu} E^+(\cdot,\gamma_\nu).
\]
Using the orthogonality we find
\begin{equation} \label{eq:C^+E}
\begin{split}
b_{\mu,\nu} \langle E^+(\cdot,\gamma_\nu),E^+(\cdot,\gamma_\nu) \rangle_{\mathbf t}^+ &= \langle E^+(\cdot,\gamma_\nu), C^+E(\cdot,\gamma_\mu) \rangle_{\mathbf t}^+ \\
&= \langle V^+E^+(\cdot,\gamma_\nu), c_+(x)E(\cdot,\gamma_\mu) \rangle_{\mathbf t}^+ \\
&= \langle E^+(\cdot,\gamma_\nu), E(\cdot,\gamma_\mu)\big) \rangle_{\mathbf t}.
\end{split}
\end{equation}
Since $E^+(\cdot,\gamma_\nu) \in \mathcal P(\nu)^{W_0}$ we obtain from the orthogonality relations for the nonsymmetric Wilson polynomials that $b_{\mu,\nu}=0$ if $\mu \not\in W_0\nu$, i.e., if $\nu \neq \lambda$. We can find $b_{\mu,\lambda}$ by evaluating at $x_0$. Using $(wc_+)(-x_0)=0$ if $w \neq 1$, we obtain
\[
(C^+E(\cdot,\gamma_\mu))(x_0) =|W_0|^{-1}\sum_{w \in W_0} (wc_+)(-x_0) \big(wE(\cdot,\gamma_\mu)\big)(-x_0)= |W_0|^{-1}c_+(-x_0).
\]
This proves the proposition.
\end{proof}

\subsection{Quadratic norms}
Next we derive the quadratic norms for the symmetric Wilson polynomials, see \cite[Theorem 7.4]{D2}.
\begin{thm} \label{thm:norm+}
For $\lambda \in \Lambda^+$,
\[
\langle E^+(\cdot,\gamma_\lambda), E^+(\cdot,\gamma_\lambda) \rangle_{\mathbf t}^+ = \frac{1}{N_+(\gamma_\lambda)},
\]
where
\[
N_+(\gamma_\lambda)=\frac1{\langle 1,1 \rangle_{\mathbf t}^+} \frac{c_+^\sigma(-\gamma_\lambda)}{c_+^\sigma(-\gamma_0)} \prod_{\alpha \in \mathcal R_r^+ \cap \tau(\lambda) \mathcal R_r^-} \frac{ c_{-\alpha}^\sigma(\gamma_\lambda) }{ c_{\alpha}^\sigma(\gamma_\lambda) }
\]
\end{thm}
\begin{proof}
From \eqref{eq:C^+E} and Theorem \ref{thm:norm} we obtain
\[
\begin{split}
\langle E^+(\cdot,\gamma_\lambda), E^+(\cdot,\gamma_\lambda) \rangle_{\mathbf t}^+
& = \frac{|W_0|}{c_+(-x_0)}\langle E^+(\cdot,\gamma_\lambda), E(\cdot,\gamma_\lambda) \rangle_{\mathbf t} \\
& = \frac{|W_0|d_{\lambda,\lambda}}{c_+(-x_0) N(\gamma_\lambda)},
\end{split}
\]
where $d_{\lambda,\lambda}$ is the coefficient of $E(\cdot,\gamma_\lambda)$ in the expansion of $E^+(\cdot,\gamma_\lambda)$ from Proposition \ref{prop:P(lambda)W}, i.e.,
\[
d_{\lambda,\lambda} = \frac{c_+^\sigma(-\gamma_\lambda)}{c_+^\sigma(-\gamma_0)}.
\]
Writing out $N(\gamma_\lambda)$ and using Lemma \ref{lem:inprod} gives the result.
\end{proof}


\begin{thebibliography}{99}

\bibitem{Ch1} I. Cherednik, \textit{Inverse Harish-Chandra transform and difference operators}, Int. Math. Res. Not \textbf{1997}, no. 15, 733-750.

\bibitem{Ch2} I. Cherednik, \textit{Double affine Hecke algebras}, London Mathematical Society Lecture Note Series, \textbf{319}, Cambridge University Press, Cambridge, 2005.

\bibitem{D1} J.F. van Diejen, \textit{Multivariable continuous Hahn and Wilson polynomials related to integrable difference systems}, J. Phys. A \textbf{28} (1995), no. 13, L369-L374.

\bibitem{D2} J.F. van Diejen, \textit{Properties of some families of hypergeometric orthogonal polynomials in several variables}, Trans. Amer. Math. Soc. \textbf{351} (1999),  no. 1, 233-270.

\bibitem{EGO} P. Etingof, W.L. Gan, A. Oblomkov, \textit{Generalized double affine Hecke algebras of higher rank}, J. Reine Angew. Math. \textbf{600} (2006), 177-201.

\bibitem{EG} P. Etingof, V. Ginzburg, \textit{Symplectic reflection algebras, Calogero-Moser space, and deformed Harish-Chandra homomorphism}, Invent. Math. \textbf{147} (2002), no. 2, 243-348.

\bibitem{ER1} P. Etingof, E. Rains, \textit{New deformations of group algebras of Coxeter groups}, Int. Math. Res. Not. \textbf{2005}, no. 10, 635-646

\bibitem{ER2} P. Etingof, E. Rains, \textit{New deformations of group algebras of Coxeter groups, II}, math.QA/0604519.

\bibitem{GG} W.L. Gan, V. Ginzburg, \textit{Deformed preprojective algebras and symplectic reflection algebras for wreath products}, J. Algebra \textbf{283} (2005), no. 1, 350-363.

\bibitem{G} W. Groenevelt, \textit{Fourier transforms related to a root system of rank 1}, Transform. Groups \textbf{12} (2007), no. 1, 77-116.

\bibitem{Gu} R.A. Gustafson, \textit{A generalization of Selberg's beta integral}, Bull. Amer. Math. Soc. (N.S.) \textbf{22} (1990), no. 1, 97-105.

\bibitem{H} G.J. Heckman, \textit{Root systems and hypergeometric functions. II}, Compositio Math. \textbf{64}(1987), no. 3, 353-373.

\bibitem{K}  T.H. Koornwinder, \emph{Askey-Wilson polynomials for root systems of type $BC$}, Contemp. Math., \textbf{138} (1992), 189-204.

\bibitem{M} I.G. Macdonald, \textit{Affine Hecke algebras and orthogonal polynomials}, Cambridge Tracts in Mathematics, \textbf{157}, Cambridge University Press, Cambridge, 2003.

\bibitem{NUKW} A. Nishino, H. Ujino, Y. Komori, M. Wadati, \textit{Rodrigues formulas for the non-symmetric multivariable polynomials associated with the $BC\sb N$-type root system},  Nuclear Phys. B \textbf{571} (2000), no. 3, 632-648.

\bibitem{N} M. Noumi, \textit{Macdonald-Koornwinder polynomials and affine Hecke rings}, Surikaisekikenkyusho Kokyuroku  \textbf{919} (1995), 44-55, (in Japanese).

\bibitem{O} E.M. Opdam, \textit{Harmonic analysis for certain representations of graded Hecke algebras}, Acta Math. \textbf{175} (1995), no. 1, 75-121.

\bibitem{S1} S. Sahi, \textit{Nonsymmetric Koornwinder polynomials and duality},  Ann. of Math. (2) \textbf{150} (1999),  no. 1, 267-282.

\bibitem{S2} S. Sahi, \textit{Some properties of Koornwinder polynomials},  $q$-series from a contemporary perspective (South Hadley, MA, 1998),  395-411, Contemp. Math., \textbf{254}, Amer. Math. Soc., Providence, RI, 2000.

\bibitem{St} J.V. Stokman, \textit{Koornwinder polynomials and affine Hecke algebras},  Int. Math. Res. Not. \textbf{2000},  no. 19, 1005-1042.

\bibitem{SK} J.V. Stokman, T.H. Koornwinder, \textit{Limit transitions for $BC$ type multivariable orthogonal polynomials}, Canad. J. Math. \textbf{49} (1997), no. 2, 373-404.
    
\bibitem{W} J.A. Wilson, \textit{Some hypergeometric orthogonal polynomials},  SIAM J. Math. Anal. \textbf{11} (1980), no. 4 , 690-701.

\bibitem{Z} G. Zhang, \textit{Spherical transform and Jacobi polynomials on root systems of type BC}, Int. Math. Res. Not. \textbf{2005}, no. 51, 3169-3189.
\end{thebibliography}
\end{document}